\documentclass[12pt]{amsart}

\usepackage[a4paper,margin=2.5cm]{geometry}
\usepackage{amsmath, amssymb, latexsym, amsfonts,amsthm, xcolor, url, mathtools}
\usepackage{multirow}
\usepackage{enumitem}
\usepackage{adjustbox}
\usepackage{makecell}
\usepackage{changepage}

\usepackage{mathtools}

\usepackage{tikz}
\usetikzlibrary{calc}

\usepackage[backref = page]{hyperref}
\hypersetup{
    colorlinks=true,
    linkcolor=blue,
    urlcolor =blue, 
    citecolor = magenta}

\renewcommand*{\backrefalt}[4]{%
\ifcase #1 %
{\color{red} No citations.}%
\or
(p.~#2).%
\else
(pp.~#2).%
\fi
}

\usepackage[normalem]{ulem}

\setlist[enumerate]{label={\rm(\roman*)}}

\usepackage[nameinlink, capitalize, noabbrev]{cleveref}
\crefname{type}{type}{types}

\usepackage{float}

\theoremstyle{plain}
\newtheorem{Thm}{Theorem}[section]
\newtheorem*{Prob}{Asymptotics Problem}
\newcounter{problemcounter}
\crefformat{problem}{#2Asymptotics Problem#3}
\crefformat{proplabel}{#2Proposition \ref*{galoisprop}#1#3}

\newtheorem*{Thm*}{Theorem}
\newtheorem{Cor}[Thm]{Corollary}
\newtheorem{Prop}[Thm]{Proposition}
\newtheorem{Lem}[Thm]{Lemma}

\theoremstyle{definition}

\newtheorem{Ex}[Thm]{Example}

\newtheorem{remarkx}[Thm]{Remark}
\newenvironment{rk}
  {\pushQED{\qed}\begin{remarkx}}
  {\popQED\end{remarkx}}

\usepackage{listings}
\lstdefinestyle{mystyle}{
	basicstyle=\ttfamily\footnotesize,
	breakatwhitespace=false,         
	breaklines=true,                 
	keepspaces=true,                 
	numbers=none,                    
	showspaces=false,                
	showstringspaces=false,
	showtabs=false,                  
	tabsize=2
}
\lstdefinelanguage{gp}{
basicstyle=\color{gray},
  keywords={sqrt, log, Euler, idealprimedec, sum, forprime, abs, poldegree, factor, Mod, apply, poldisc},
   keywordstyle=\color{blue},
  comment=[l]{\\},
  commentstyle=\color{green!50!black}\ttfamily,
    basicstyle=\ttfamily\footnotesize,
    numbers=left,
    numberstyle=\scriptsize
}
\lstdefinelanguage{gp2}{
basicstyle=\color{gray},
  keywords={sqrt, log, Euler, idealprimedec, sum, forprime, abs, poldegree, factor, Mod, apply, poldisc},
   keywordstyle=\color{blue},
  comment=[l]{\\},
  commentstyle=\color{green!50!black}\ttfamily,
    basicstyle=\ttfamily\scriptsize,
    numbers=left,
    numberstyle=\scriptsize
}
\lstdefinelanguage{gpin}{ 
    basicstyle=\ttfamily\footnotesize,
  comment=[l]{\\},
  commentstyle=\color{green!50!black}\ttfamily,
    belowskip=-5pt
}
\lstdefinelanguage{gpout}{
  basicstyle = \color{blue}\ttfamily\footnotesize,
}
\lstdefinelanguage{gpoutmsg}{
  basicstyle = \color{black}\ttfamily\footnotesize,
}

\usepackage{accsupp}

\newcommand{\Gal}{\mathrm{Gal}}

\newcommand{\Frob}{\mathrm{Frob}}

\DeclareMathOperator{\charpoly}{{\rm CharPoly}}
\DeclareMathOperator{\Ind}{{\rm Ind}}

\newcommand{\Q}{\mathbb{Q}}
\newcommand{\lam}{\lambda}
\newcommand{\CC}{\mathbb C}
\newcommand{\QQ}{\Q}
\newcommand{\Z}{\mathbb{Z}}
\newcommand{\ZZ}{\Z}
\newcommand{\p}{\mathfrak{p}}
\newcommand{\pp}{\p}
\newcommand{\OO}{\mathcal O}

\renewcommand{\SS}{{\mathfrak S}}

\usepackage{bbm}
\newcommand{\one}{\mathbbm{1}}
\newcommand{\sgn}{\mathrm{sgn}}
\newcommand{\std}{\mathrm{Std}}
\newcommand{\twod}{\mathrm{2D}}

\newcommand{\CCC}{\mathcal{C}}

\newcommand{\rhobar}{\overline{\rho}}
\newcommand{\FF}{\mathbb F}
\newcommand{\Vbar}{\overline{V}}
\DeclareMathOperator{\GL}{{\rm GL}}
\DeclareMathOperator{\PGL}{{\rm PGL}}
\DeclareMathOperator{\tr}{{\rm tr}}

\DeclareMathOperator{\Res}{Res}
\DeclareMathOperator{\Reg}{Reg}
\newcommand{\NR}{{\rm NR}}
\newcommand{\QR}{{\rm QR}}

\newcommand{\pmodi}[1]{%
  \mathchoice%
    {\mkern4mu(\mathrm{mod} \, #1)}
    {\mkern4mu(\mathrm{mod} \, #1)}
    {\mkern4mu(#1)}
    {\mkern4mu(#1)}
}

\def\+{{\mkern2mu+\mkern2mu}}
\def\-{{\mkern2mu-\mkern2mu}}

\begin{document}
\title[Euler-Kronecker constants beyond Dirichlet $L$-series]{Euler-Kronecker constants of modular forms: \\ beyond
Dirichlet $L$-series}

\author[Charlton]{Steven Charlton}
\address{Max Planck Institute for Mathematics, Vivatsgasse 7, 53111 Bonn, Germany}
\email{charlton@mpim-bonn.mpg.de}
\email{moree@mpim-bonn.mpg.de}

\author[Medvedovsky]{Anna Medvedovsky}
    \address{Department of Mathematics, University of Arizona, 617 N.~Santa Rita Ave, Tucson, AZ 85721, USA}
    \email{medved@arizona.edu}

\author[Moree]{Pieter Moree}

\date{}

\begin{abstract}
\noindent The Euler-Kronecker constants related to congruences of Fourier coefficients of modular forms that have been computed so far, involve logarithmic derivatives of Dirichlet $L$-series as most complicated functions (to the best of our knowledge). However, generically the more complicated Artin $L$-series will make their appearance. Here we work out some simple examples involving an Artin $L$-series related to an $\mathfrak{S}_3$, respectively $\mathfrak{S}_4$ extension. These examples are related to a mod-2 congruence for $X_0(11)$, respectively a mod-59 congruence for $\Delta E_4$ conjectured by Serre and Swinnerton-Dyer and proved by Haberland. The latter example solves a problem posed by Ciolan, Languasco and the third author in 2023.
\end{abstract}

\maketitle

{
\parskip=0pt           
    \setcounter{tocdepth}{1}
	\tableofcontents
    \setcounter{tocdepth}{2}
}

\section{Introduction}
An arithmetic function $a:\mathbb N\rightarrow \mathbb C$ is said to be \emph{multiplicative} if $a(mn)=a(m)a(n)$ for every pair of coprime integers $m$ and $n$.
A set $S$ of positive integers is said to be \emph{multiplicative} if for every pair $(m,n)$ of coprime 
positive integers we have
$mn\in S$ if and only if  $m,n \in S$.  (In other words, $S$ is a multiplicative set if
and only if the indicator  function $i_S$ of $S$ is multiplicative.) 
An enormous supply of multiplicative sets is provided 
by taking 
\begin{equation}
\label{Sdef}
S:=S_{a, \ell}:=\{n\ge 1:\ell\nmid a(n)\},
\end{equation}
where $a$ is any integer-valued multiplicative function and 
$\ell$ a prime. 
It is natural to ask for the asymptotic behavior of $S(x)$,  
the number of positive integers $n\le x$ that are in~$S$.
An important role in answering this 
question is played by the 
\emph{Dirichlet generating series} (which converges for $\Re(s) > 1$) 
\begin{equation}
\label{LS}
F(s) \coloneqq F_S(s) \coloneqq \sum_{n\in S}\frac1{n^{s}}=\sum_{n=1}^{\infty}
\frac{i_S(n)}{n^s},
\end{equation}
where $i_S$ is the \emph{characteristic} 
(\emph{indicator}) function of $S$.
One of the basic facts about a formal Dirichlet series $G(s):=\sum_{n=1}^{\infty}g(n)n^{-s}$ with $g$ a multiplicative arithmetic function, is that it has an \emph{Euler product}:
$$G(s) = \prod_{p\ {\rm prime}} G_p(s) = \prod_{p\ {\rm prime}}\;\; \sum_{k\ge 0} \frac{g(p^k)}{p^{ks}}.$$
In particular, if $S$ is a multiplicative set, then
\begin{equation}
\label{eulerpro}    
F(s) = \prod_{p \text{~prime}} F_p(s), \qquad \text{where}
\qquad
F_p(s) = \sum_{\substack{k \geq 0 \\ p^k \in S}} \frac{1}{p^{ks}} = \sum_{k \geq 0} \frac{i_S(p^k)}{p^{ks}}. 
\end{equation}
In the sequel the mathematical symbol $p$ will be exclusively used to denote primes.

The shorthand $F'/F(s)$ stands for $F'(s)/F(s)$.
If the limit
\begin{equation}
\label{EKf} 
\gamma_S:=\lim_{s\rightarrow 1^+}\bigg(
\frac{F'}{F}(s) 
+\frac{\alpha}{s-1}\bigg)
\end{equation}
exists for some $\alpha>0$, we say that the set $S$ admits an \emph{Euler-Kronecker constant} $\gamma_S$. In case
$S=\mathbb N,$ we have 
$F(s)=\zeta(s)$, the 
Riemann zeta function, $\alpha=1$ and $\gamma_S=\gamma$, the \emph{Euler-Mascheroni 
constant}.
The following result shows
that the Euler-Kronecker constant~$\gamma_S$ determines the second order behavior of $S(x).$ As usual $\Gamma$ denotes the Gamma function.
\begin{Thm*}[Standard]
\label{thm:multiplicativeset}
Let $S$ be a multiplicative set.
If there exist real numbers $\rho>0$ and $0<\delta<1$ such that
\begin{equation}
\label{primecondition}
\sum_{p\le x,~p\in S}1=\delta~\sum_{p\le x}1+O_S\bigg(\frac{x}{\log^{2+\rho}x}\bigg),
\end{equation}
then $\gamma_S\in\mathbb R$ exists and 
\begin{equation}
\label{initstarrie}
S(x)=\sum_{n\le x,~n\in S}1=\frac{C_S\,x}{\log^{1-\delta}x}\bigg(1+\frac{(1-\gamma_S)(1-\delta)}{\log x}(1+o_S(1))\bigg)
\end{equation}
as $x\to\infty,$ where 
\begin{equation}
\label{residu}    
C_S=\frac{1}{\Gamma(\delta)}\lim_{s\rightarrow 1+}(s-1)^{\delta}F(s)>0.
\end{equation}
If  the prime numbers belonging to $S$ are, with finitely many exceptions, precisely
those in a finite union of arithmetic progressions, 
we have, for arbitrary $j\ge 1,$
\begin{equation}
\label{starrie}
S(x)=\frac{C_S\,x}{\log^{1-\delta}x}\left(1+\frac{c_1}{\log x}+\frac{c_2}{\log^2 x}+\dots+
\frac{c_j}{\log^j x}+O_{j,S}\left(\frac{1}{\log^{j+1}x}\right)\right),
\end{equation}
with $C_S,c_1,\ldots,c_j$ constants, 
$C_S$ as in {\eqref{residu}} and $c_1=(1-\gamma_S)(1-\delta)$.  
\end{Thm*}
There are many results in this spirit in the literature. For \eqref{initstarrie}, see Moree \cite[Theorem 4]{Mpreprint}; for the remaining assertions the reader can consult, for example, Serre \cite[Th\'eor\`eme~2.8]{serre}.

In order for the standard theorem  to apply to a set $S$ of the form \eqref{Sdef}, 
it is necessary that the set of primes $p$ 
for which
$q$ divides $a(p)$ 
has a positive natural density $\delta$. If this set of primes is well-behaved enough analytically (cf.\,the assumption \eqref{primecondition}), then the $c_j$ will be related to the coefficients of the Laurent expansion of $F(s)$ around $s=1$ (see Serre \cite[Th\'eor\`eme 2.8]{serre}). We will exclusively focus on $C_S$ and $c_1$, which 
are related to $F(s)$ by \eqref{residu}, respectively \eqref{EKf}.

\noindent There are many Dirichlet series having an associated Euler-Kronecker constant, we refer to Moree~\cite{India} for a number of examples.
Here we exclusively focus on the case where $S$ is of the form~\eqref{Sdef} and the $a(n)$ are integer valued Fourier coefficients of a modular form, where we require~$f$ to be multiplicative. 
The first to 
be interested in estimating
$S(x)$ was Ramanujan who took
$\ell\in \{3,5,7,23,691\}$ and $a(n)=\tau(n)$, with $\tau$ the \emph{Ramanujan tau function}. 
He did so
in his ``unpublished manuscript" \cite{bono} (published only decades after he passed away). It addresses many topics, however the  material relevant for us is summarized in Rankin \cite{Rankin76}. Recall that
\begin{equation}
    \label{productexpansion}
q\prod_{n=1}^{\infty}(1-{q}^n)^{24}=\sum_{n=1}^{\infty}\tau(n){q}^n,\quad q:=e^{2\pi i z}
\end{equation}
the associated cusp form being
$\Delta=\frac{1}{1728}
(E_4^3-E_6^2),$
with
$E_4$ and $E_6$ being the 
\emph{normalized Eisenstein series}
$$E_4=1+240\sum_{n\ge 1}\sigma_3(n)q^n,\quad E_6=1-504\sum_{n\ge 1}\sigma_5(n)q^n.$$
We use the notation $$\sigma_r(n):=\sum_{d\mid n}d^r,$$ to denote the \emph{sum of divisors function}.
For each of Ramanujan's primes $\ell$ there is a congruence involving some sum of divisors function, 
allowing us to replace the potentially mysterious Fourier coefficients in \eqref{Sdef} by 
some simple arithmetic function.
The primes $p$ with $\ell\nmid \sigma_r(p)$ and $p\nmid \ell$, are precisely all the the primes in a finite union of arithmetic progressions modulo $\ell$, and hence $S(x)$ satisfies \eqref{starrie}.
Ramanujan \cite{bono}
claimed that $S(x)$ satisfies
\begin{equation}
\label{boldclaim}    
S(x)=C_S\int_2^x 
\frac{dt}{\log^{1-\delta}t}+O_{\epsilon}(x^{1/2+\epsilon}),{\rm ~with~}\epsilon>0{\rm~arbitrary},
\end{equation}
and indicated explicit values of $C_S$ and $\delta$, that 
were later proven to be correct 
by Rankin (cf.\,\cite[Sec.\,5]{CLM}). 
Partial integration of \eqref{boldclaim} yields
$$S(x)=C_S\frac{x}{\log^{1-\delta}t}\left(1+
\frac{1-\delta}{\log x}+O\left(\frac{1}{\log^2 x}\right)\right),$$
and thus $c_1=1-\delta$. On the other hand, 
by \eqref{starrie} we have 
$c_1=(1-\gamma_S)(1-\delta)$. Thus Ramanujan's conjecture implies $\gamma_S=0$.
However, Moree \cite{MRama} showed that $\gamma_S\ne 0$ in each of Ramanujan's 
cases. In some of these cases though the integral in \eqref{boldclaim} asymptotically better approximates $S(x)$ than $C_S\,x\log^{\delta-1}x$ (note that this is guaranteed if $\gamma_S<1/2$ and hence $(1-\gamma_S)(1-\delta)$ is closer to $1-\delta$ than to $0$).
All of the above is recounted in greater detail in Moree and Cazaran \cite{MC} and again much more recently in Berndt and Moree \cite{bemo}.

In the hands of Deligne and Serre, Ramanujan's work has found vast generalization, of which
we will only consider a rather specific case. 

Recall that a set 
$S$ of rational primes is called \emph{frobenian} 
if there exists a finite Galois extension $L/\mathbb Q$ and a union of conjugacy classes $\mathcal C$ of 
$\Gal(L/\mathbb Q)$, such that the symmetric difference between 
$S$ and the set $\{p\text{~prime}:p\text{~is~unramified~in~}L/\mathbb Q\text
{~and~}\Frob_p\in \mathcal C\}$ is finite. We will also call such a set \emph{$L$-frobenian}. By the Chebotarev density theorem a frobenian set $S$ has a rational density $\delta_S\in [0,1]$. 

\begin{Thm*}[Deligne, Serre \cite{Deligne,DeligneSerre}]
Given a cuspidal eigenform  $f(z)=\sum_{n=1}^{\infty}a_n(f)\,q^n$ with integral coefficients and a prime $\ell$, the set
$$S=S_{f,\ell}=\{n:\ell\nmid a_n(f)\}$$
is frobenian. The asymptotic behavior of $S(x)$  is given  by \eqref{initstarrie}, with $\delta=\delta_{S}$, where ~$\delta_{S}$ is the density of $S$.\footnote{In fact, this statement is true if~$f$ is a general modular form, not necessarily an eigenform (already known to Serre in the 70s, or see \cite[\S 11.1]{Bim}). On the other hand, the set $S_{f, \ell}$ is multiplicative if and only if $f$ is a normalized eigenform \cite[Prop.~5.8.5]{DS}.}
\end{Thm*}

Inspired by \eqref{initstarrie} and Ramanujan's claims of the form \eqref{boldclaim}, one can raise the following

\begin{Prob}%
\refstepcounter{problemcounter}
\label[problem]{problem}
 Determine $C_S$ and $\gamma_S$ with 
several certified decimal digits.
\end{Prob}

Ramanujan's congruences suggest that it is worthwhile to consider $S=\{n\ge 1:\ell\nmid \sigma_r(n)\}$, with $\ell$ an arbitrary prime and $r$ an arbitrary positive integer.
Ramanujan wrote down some initial thoughts, see \cite{Rankin76}. Later 
Rankin \cite{Rankin} computed $C_S$ and $\delta$.
Ciolan et al.\,\cite{CLM} 
determined
$\gamma_S$ and thus determined $c_1$ in \eqref{starrie}. They applied their results to compute the associated $\gamma_S$ for congruences involving the generators of the space of cusp forms of weight
$k\in \{12,16,18,20,22,26\}$, 
confirming the earlier results of Moree \cite{MRama} 
for $k=12$. We recall that precisely for these weights the space of cusp forms of the full modular group is of dimension 1. All the 
`elementary' congruences of prime modulus were 
completely classified using $\ell$-adic representations by the efforts of Serre and Swinnerton-Dyer, cf.\,\cite{S-D}, 
and could be dealt with by 
Ciolan et al.\,\cite{CLM}, with one 
exception. This concerns the
congruential restriction (or congruence, for short)
\begin{equation}
\label{Haberlandcong}    
a_p(\Delta_{16})^2\equiv 0,\ p^{15},\ 2p^{15},\ 4p^{15} \pmod{59} \quad (p\ne 59), 
\end{equation}
where $\Delta_{16}=\Delta E_4$ is the unique normalized 
cusp form of weight 16 for the full modular group. In \cref{Haberland59} we fill this gap in the literature and determine the Euler-Kronecker
constant for $S=S_{\Delta_{16},59}=\{n:~59\nmid a_n(\Delta_{16})\}$. 

As a warm-up we deal with a somewhat easier problem first, namely counting odd Fourier 
coefficients  of  $f_E=\eta(z)^2\eta(11z)^2$,  the unique normalized cuspform of weight $2$ and level $\Gamma_0(11)$, corresponding to the modular elliptic curve $X_0(11)$ 
(thus here we consider \mbox{$S=S_{f_E,2}=\{n:2\nmid a_n(f_E)\}$}). 
We do this in 
\cref{sec:parity} after having dealt with the totally standard case
$S=S_{f_E,5}$ in \cref{five}.

In all of the $\gamma_S$ computations discussed in the literature so far, $F_S(s)$, raised to an appropriate power, has Dirichlet $L$-series as most complicated functions appearing in them, cf.~\eqref{basisvergelijking}. It turns out that
this also happens for $S=S_{f_E,5}$.

The novel feature in this paper is the natural appearance of
Dedekind $\zeta$-functions $\zeta_K(s)$ for our main examples $S_{f_E,2}$ and $S_{\Delta_{16},59}$ with $K$ non-abelian (these cannot be expressed in terms
of Dirichlet $L$-series).
These examples are carefully hand-picked so as to both keep the resulting analysis as easy as possible and involve well-studied cusp forms. 
\begin{Thm}
\label{maintheorem}
Given a cusp form $f$ and $\ell$ as in the table below, for 
$$S=S_{f,\ell}=\{\,n:\ell\nmid a_n(f)\,\},$$ we have 
\begin{equation}
\label{starrieexpliciet}
\,\, S(x)=\sum_{\substack{n\in S \\ n\le x}}1=\frac{C_{S}\,x}{\log^{1-\delta_{S}}x}\left(1+\frac{(1-\gamma_{S})
(1-\delta_{S})}{\log x}+\dots+
\frac{c_{j}}{\log^j x}+O_{j}\left(\frac{1}{\log^{j+1}x}\right) \! \right), \!
\end{equation}
with the $c_{j}$ real constants, possibly depending on
both $f$ and $\ell$, and $\delta_{S},C_{S}$ and $\gamma_{S}$ as in the 
corresponding row of the table\footnote{The bracketed digits
are \emph{uncertified} and rely on \texttt{Pari/GP}'s heuristic numerical computation of $L$-functions; see \cref{app:higherprec} for more details}, where the $\gamma_S$ in the second and third row are conditional on the Generalized Riemannn Hypothesis\footnote{More precisely, RH for the Dedekind $\zeta$-functions appearing in
\eqref{eqn:fvArtinfree}, respectively \eqref{eqn:ls3:dedekind}.}.
\begin{table}[ht]
\begin{center}
\footnotesize
\renewcommand{\arraystretch}{1.25}
\begin{tabular}{|c|c|c|c|c|}\hline
$f$ & $\ell$ & $\delta_{S}$ & $C_{S}$ & $\gamma_{S}$ \\ \hline \hline
$f_E = \eta(z)^2\eta(11z)^2$ & $5$ & $3/4$ & $1.116589605788\ldots$ & \multicolumn{1}{l|}{$-0.003392959329\ldots$}\\ \hline 
$f_E = \eta(z)^2\eta(11z)^2$ & $2$ & $1/3$ & \multicolumn{1}{l|}{$0.412561\ldots$} &  \multicolumn{1}{l|}{$-0.1402(531)\ldots$} \\ \hline 
$\Delta_{16} = \Delta(z)E_4(z)$ & $59$ & $5/8$  & \multicolumn{1}{l|}{$0.913313\ldots$} & \multicolumn{1}{l|}{$\phantom{+}0.29(6104)\ldots$} \\ \hline 
\end{tabular}
\end{center}
\label{tab:table2}
\end{table}

The corresponding Dirichlet generating series $F_S(s)=\sum_{n\in S}n^{-s}$ only involve Dedekind $\zeta$-functions and are given 
in \eqref{basisvergelijkingD}, \eqref{eqn:fvArtinfree}, respectively~\eqref{eqn:ls3:dedekind}. 
\end{Thm}
We prove \cref{maintheorem}, our main computational theorem, in the \cref{five,,sec:parity,,Haberland59}, where we deal
with the rows of the table in order of appearance (where the ordering is according to increasing level of difficulty). Preliminaries on
Dedekind $\zeta$-functions and Artin $L$-series are discussed in \cref{sec:prelim}, those for the modular curve $X_0(11)$ 
(relevant for the first two rows) in \cref{sec:modularcurve}.

In dealing with the examples in \cref{maintheorem} we proceed in three steps:
\begin{itemize}[topsep = -3pt]
    \item Determine an Euler product for the generating series $F(s)$.
    \item Rewrite $F(s)$ in terms of \emph{familiar} $L$-series (double entendre intended!).
    \item Use properties of these $L$-series to answer the \cref{problem}.
\end{itemize}
For the familiar $L$-series we do not have to look beyond the $L$-series 
associated to the ``field of determination"\footnote{By the \emph{field of determination} of a modular form $f = \sum a_n(f) q^n$ mod $p$ we mean the smallest number field $K_f$ with the property that for almost all primes $\ell$ the Fourier coefficient $a_\ell(f)$ depends only on the conjugacy class of $\Frob_\ell$ in $\Gal(K_f / \QQ)$. If $f$ is an eigenform, then $K_f$ is the fixed field of the kernel of the semisimple mod-$p$ Galois representation attached by $f$ by a construction of Deligne.} of the modular form. The simpler this $L$-series is, the easier it is to answer the Asymptotics Problem.

In the proof of \cref{maintheorem} we make copious use of Artin L-series in determining 
the relevant $F_S(s)$.
Curiously it turns out that these can 
always be expressed in terms of Dedekind Zeta-functions only. 
It is thus natural to wonder whether there exists 
 a  multiplicative set of the form $S_{f,\ell}$ (that is, a set arising from a modular form $f$ and a prime $\ell$)
with an expansion of the form \eqref{Dedekindproduct} in terms of Artin $L$-series that \emph{cannot} be expressed with Dedekind $\zeta$-functions. Indeed, 
finding such a set $S_{f,\ell}$ was the original motivation for writing this paper. 
Our main theoretical result establishes that this quest will always fail. 

\begin{Thm}
\label{simpleanna}
Given a cuspidal eigenform  $f(z)=\sum_{n=1}^{\infty}a_n(f)\,q^n$ (of some weight and level) and a prime $\ell$, the generating series $F_S(s)$ of
$S=S_{f,\ell}=\{n:\ell\nmid a_n(f)\}$ satisfies
\begin{equation}
\label{Dedekindproduct}
F_S(s)=H(s)\prod_{j=1}^m \zeta_{K_j}(s)^{e_j}, 
\end{equation}
where $K_1,\ldots,K_m$ are number fields,
$e_1,\ldots,e_m$ are rational numbers 
and $H(s)$ is holomorphic and non-zero
and uniformly bounded
in \mbox{$\Re(s)>1-\varepsilon$} for some $\varepsilon>0$.
\end{Thm}
This result applies in particular to the three examples studied in 
\cref{maintheorem}.
Indeed, \eqref{basisvergelijkingD}, \eqref{eqn:fvArtinfree}, respectively \eqref{eqn:ls3:dedekind} are the corresponding explicit versions of \eqref{Dedekindproduct}.

From character theory we know, \emph{a priori}, that $F_{S_{f, \ell}}(s)$ is, up to $H(s)$, a product of $\CC$-powers of Artin $L$-functions, so \cref{simpleanna} is a surprising result.

The proof of \cref{simpleanna} can be
found in \cref{generalsection},
where a more general version is formulated and established (\cref{mainthm}).

A nice first introduction 
to the philosophy of this paper,
capturing well both the zeta function and modular spirit of our article, can be found in Murty \cite{Murtysieving}.
Finally, we point out that Serre 
\cite{serre,serreIHES} 
and, especially, Odoni \cite{Odoni1,Odoni2}
considered the problem of estimating~$S(x)$ for multiplicative Frobenian sets $S$ in much great generality than we do, however, we do so in a much more explicit way.

\subsection*{Acknowledgements}
The authors thank Neil Dummigan, Jan Vonk, 
Gabor Wiese and Wadim Zudilin for helpful email correspondence and 
Robert van der Waall for sending an extensive letter.
Special thanks are due to Alessandro Languasco for 
extensive feedback and checking various of our computations independently.
The bulk of this paper was written whilst all authors were employed by the Max Planck Institute for Mathematics in Bonn. They are grateful for this opportunity and thank the staff for their helpfulness.

\section{Preliminaries}
\label{sec:prelim}
Our paper requires an excursion to the Zeta Zoo.
In \cref{sec:Dedekind,sec:Artin} we 
briefly discuss some of its animals with a focus on their Euler-Kronecker constants
(for more material the
reader can consult 
Murty \cite{RamMurty}, for example). 
\subsection{Dedekind \texorpdfstring{$\zeta$}{zeta}-functions and their Euler-Kronecker constants}
\label{sec:Dedekind}
Let $K$ be a number field, $\mathcal{O}$ its ring of integers and $s$ a complex variable.
Then for Re$(s)>1$ the 
\emph{Dedekind $\zeta$-function} is defined as
\begin{equation}
\label{KEuler}
\zeta_K(s)=\sum_{\mathfrak{a}} \frac{1}{N{\mathfrak{a}}^{s}}
=\prod_{\mathfrak{p}}\frac{1}{1-N{\mathfrak{p}}^{-s}},
\end{equation}
where $\mathfrak{a}$ ranges over 
the non-zero ideals in $\mathcal{O}$, 
$\mathfrak{p}$ ranges over the prime ideals in $\mathcal{O}$, and~$N{\mathfrak{a}}$ denotes the 
\emph{absolute norm}
of $\mathfrak{a}$,
that is the index of $\mathfrak{a}$ in $\mathcal{O}$.
It is known that $\zeta_K(s)$ can be analytically continued to ${\mathbb C} \setminus \{1\}$,
and that it has a simple pole at~$s=1$. 
Further, around $s=1$ we have the Laurent expansion
\begin{equation}
\label{laurent}
\zeta_K(s)=\frac{\alpha_K}{s-1}+c_K+c_1(K)(s-1)+c_2(K)(s-1)^2+\cdots,
\end{equation}
where $\alpha_K$ is a positive real number carrying a lot of arithmetic information
on $K$.
 The ratio $\gamma_K=c_K/\alpha_K$ is called the \emph{Euler-Kronecker constant}.
In particular, 
we have $\gamma_{\mathbb Q}=\gamma$.
An alternative definition of $\gamma_K$ is 
\begin{equation}
\label{define}
\gamma_K=\lim_{
s \to 1^+
}\,\Bigl(\frac{\zeta'_K(s)}{\zeta_K(s)}+\frac{1}{s-1}\Bigr),
\end{equation}
which shows that $\gamma_K$ is the constant part in the Laurent series of the logarithmic derivative
of $\zeta_K(s)$. 
The following expression can be used to obtain a quick rough approximation
of~$\gamma_K$:
\begin{equation}
\label{hickup}
\gamma_K=\lim_{x\rightarrow \infty}\Bigl(\log x - \sum_{N\mathfrak{p}\le x}\frac{\log N\mathfrak{p}}{N\mathfrak{p}-1}\Bigr).
\end{equation}
For a proof see, e.g., Hashimoto et al.\,\cite{HIKW}.
Hashimoto et al.~\cite{HIKW} (cf.\,\,Ihara \cite[pp. \mbox{416--421}]{Ihara})
show that
$$\gamma_K^{\phantom{p}}=\sum_{\rho}\frac1{\rho}-r_K - 1,$$
where the sum is over the zeros of $\zeta_K(s)$ in the critical strip and 
\begin{equation}
\label{rk}
    r_K= \frac{1}{2} \log|d_K| - \Big\{ \frac{r_1}{2} (\gamma + \log 4\pi) + r_2 (\gamma + \log 2 \pi ) \Big\}, \\
\end{equation}
with $d_K$ the absolute discriminant of $K$, $r_1$ and $2r_2$ the number of real, respectively complex places of $K$.

For later computation of the Euler-Kronecker constants of the the multiplicative sets \( S \), we need to determine various \( \gamma_K \) to sufficient accuracy to obtain provably correct decimal digits of the result.  
This seems a surprisingly difficult problem in case $K$ is non-abelian and the best approach we are aware of assumes GRH and is due
to Ihara \cite{Ihara}.

Let \( K \) be a number field, with \( r_1 \) real, \( r_2 \) imaginary places, and discriminant \( d_K \).  Consider the prime counting function, defined in 
\cite[(0.18)]{Ihara}, for \( x > 1 \),
\[
  \Phi_K(x) \coloneqq \frac{1}{x-1} \sum_{
  \substack{
  N\pp^k \leq x  \\
  k = 1, 2, 3, \ldots
  }
  } \Big( \frac{x}{N\pp^k} - 1 \Big) \log N\pp \,.
\]
Note that $\Phi_K(x)$ is roughly equal to the sum in \eqref{hickup}.
As in \cite[ \S1.2]{Ihara}, define
$$
    \ell_K(x) = \frac{r_1}{2} \Big( \log\frac{x+1}{x-1} + \frac{2}{x-1} \log\frac{x+1}{2} \Big) + r_2 \Big( \log\frac{x}{x-1} + \frac{\log x}{x-1} \Big).
$$

Ihara established the following bounds on \( \gamma_K \), under GRH for \( \zeta_K(s) \).
\begin{Prop}[{Ihara, \cite[Proposition 2, p.\,18]{Ihara}}]\label{prop:ihara}
For a number field \( K \), under GRH for \( \zeta_K(s) \), for $x>1$ we have
\begin{equation}\label{eqn:ihara}
\begin{aligned}[c]
    \gamma_K & \leq \frac{\sqrt{x}+1}{\sqrt{x}-1} \big( \log x - \Phi_k(x) + \ell_K(x) \big) + \frac{2\,r_K}{\sqrt{x}-1}  - 1 \\ 
    \gamma_K & \geq \frac{\sqrt{x}-1}{\sqrt{x}+1} \big( \log x - \Phi_k(x) + \ell_K(x) \big) - \frac{2\,r_K}{\sqrt{x}+1}  - 1, 
    \end{aligned}
\end{equation}
with $r_K$ as in \eqref{rk}.
\end{Prop}
As the difference between the bounds tends to 0, as \( x \to \infty \), one has a method for computing \( \gamma_K \) to any proven accuracy desired (unfortunately the convergence turns out to be slow).  Note also that Ihara's bounds are not monotonic.

For the convenience of the reader, we give a basic implementation of these bounds in the \texttt{Pari/GP} calculator in \cref{app:iharacode}.  The relevant command that we implement is \mbox{\texttt{iharaBounds(x, nf)}}; usage instructions can be found in \cref{app:usage}.  We also note that \texttt{Pari/GP} can heuristically compute values of \( \gamma_K \), and other special values of derivatives of $L$-functions (including Artin $L$-functions and Dedekind $\zeta$-functions) to much greater precision, in an \emph{uncertified} manner\footnote{See \textbf{Important Caveat} in the documentation on\newline\url{https://pari.math.u-bordeaux.fr/dochtml/html-stable/_L_minusfunctions.html}.}, details of which can be found in \cref{app:higherprec}.  We may often utilize or state the uncertified results, but will always emphasise this by writing \( \overset{?}{=} \) and {\small \( \smash{ \overset{?}{\in}} \)}, where necessary.

\begin{Ex}
    For the field \( E \) with defining polynomial \( f_E(X) = X^4 + X^3 - 7X^2 - 11X + 3 \), Ihara's bounds \eqref{eqn:ihara} for \( x=10^8 \), computed via \texttt{iharaBounds(x, nf)} in \cref{app:iharacode}, give
    \begin{equation}\label{eqn:rk:gammakint}
         \gamma_E \in (0.83356048\ldots, 0.83453128\ldots) \,,
    \end{equation}
    so that \( \gamma_E = 0.83\ldots \), accurate to 2 decimal places.  Via the heuristic routines described in \cref{app:higherprec}, one obtains
    \begin{equation}\label{eqn:rk:gammakintprec}
        \gamma_E \overset{?}{=} 0.8339340133907066887364345\ldots \,,
    \end{equation}
    with significantly higher precision, although \emph{uncertified}.  One notes that the uncertified value in \eqref{eqn:rk:gammakintprec} is indeed contained within Ihara's bounds from \eqref{eqn:rk:gammakint}.
\end{Ex}

\subsection{Artin \texorpdfstring{$L$}{L}-series}
\label{sec:Artin}
We now introduce a key technical tool of our work. 
A standard reference for this material is the book edited by Cassels and Fr\"ohlich \cite{Brighton}; or see Cogdell~\cite{cogdell} for a convenient summary.
Let $K$ be a number field and $$\rho: \Gal(\overline K/K) \to \GL_\CC(V) = \GL_n(\CC)$$
be an \emph{Artin  representation}: a continuous  Galois representation on a finite-dimensional complex vector space. Because of the mismatch between the profinite topology on the Galois group and the complex topology on the vector space, an Artin representation always has finite image, and hence factors through some finite Galois extension $E$ of $K$; here $E$ may be any number field containing the fixed field of $\ker \rho.$\footnote{In particular, the top field $E$ plays no role, which is why it need not appear in the notation.}

Associated to an Artin representation as above we have an \emph{Artin $L$-function} $L_K(\rho, s)$, which is 
defined as the Euler product
$$L_K(\rho, s) := \prod_{\text{$p$ prime}} L_{K, p}(\rho, s) =\prod_{\text{$p$ prime}}\prod_{\pp \mid p} L_{K, \pp}(\rho, s),$$
where in the last product $\pp$ ranges over the prime ideals of $K$ lying over the rational prime~$p$, and the Euler factor at $\pp$ is \begin{equation}
    \label{Artinlocalfactor}
L_{K, \pp}(\rho, s) \coloneqq \frac{1}{\charpoly(\Frob_\pp \mid V^{I_\pp})(N\pp^{-s})} .
\end{equation}
Here $V^{I_\pp}$ is the subspace of $V$ on which an inertia subgroup $I_\pp$ acts trivially, and the characteristic polynomial of an operator $T$ acting on a $\CC$-vector space $W$ is normalized to be 
 $$\charpoly(T \mid W) = \det\big(1 - X (T\mid W)\big) \in \CC[X].$$
Note that the inertia subgroup $I_\pp$ is defined only up to conjugation, so  $\Frob_\pp$ is therefore a conjugacy class of $I_\pp$-cosets. 
However, since $\Frob_\pp$ is well defined on $V^{I_\pp}$  up to conjugacy, the characteristic polynomial of its action on $V^{I_\pp}$ is indeed well defined.

We also need to recall a few properties of Artin-$L$ series, their relation to Dedekind $\zeta$-functions, and how they behave under induction of representations.

\begin{itemize}
    \item {\bf Connection with Dedekind $\zeta$-functions:} For a Galois extension \( E/K \), the trivial representation \( \rho = \one \) of \( \Gal(E/K) \) gives the Dedekind $\zeta$-function of the base-field
\[
    L_K(\one, s) = \zeta_K(s) \,,
\]
\item {\bf Connection with Dirichlet $L$-series:}
If $\chi$ is a Dirichlet character, and $\rho_\chi$ is the corresponding character of $\Gal(\overline\QQ/\QQ)$, then the Artin $L$-series is the same as the Dirichlet $L$-function. 
\[
    L_\QQ(\rho_\chi, s) = L(\chi, s) \,,
\]
\item {\bf Additivity:} The Artin $L$-function of a direct sum of representations is the product of the Artin $L$-functions: 
\[
    L_K(\rho_1 + \rho_2, s) =     L_K(\rho_1, s)    L_K(\rho_2, s) \,,
\]

    \item {\bf Induced representation:} Given a Galois extension \( E/K \), with \( K \subseteq M \subseteq E \), setting \( G = \Gal(E/K) \) and \( H = \Gal(E/M) \subset G \), then for a representation \( \psi \) of \( \Gal(E/M) \), we have,
    \[
        L_M(\psi, s) = L_K( \Ind_H^G \psi, s) \,.
    \] 
\end{itemize}
In particular, for a finite Galois extension $E/K$ with Galois group $G = \Gal(E/K)$ by considering the regular representation of $G$ it follows that  \begin{equation}\label{eq:regularrep}\zeta_E(s) = \prod_{\rho} L_K(\rho, s)^{\dim \rho},
\end{equation}
where~$\rho$ runs over the irreducible representations of $G$.

\section{The modular curve \texorpdfstring{$X_0(11)$}{X\_0(11)}}
\label{sec:modularcurve}

Already Fricke and Klein \cite[\S 10, 431--434]{Klein} considered the coefficients~$a_n$ in the expansion 
\begin{equation}
\label{cndefi}
f_E(z)=q\prod_{n=1}^{\infty}(1-{q}^n)^2(1-{q}^{11n})^2=
\sum_{n=1}^{\infty}a_n{q}^n.
\end{equation}
We have
$$f_E(z)=q - 2 {q}^{2} -  {q}^{3} +  2 {q}^{4} +  
{q}^{5} + 2 {q}^{6} - 
2 {q}^{7} - 2 {q}^{9} - 2 {q}^{10} +  O({q}^{11}).$$
Putting $q=e^{2\pi iz}$, we see that, for $z$ in the upper half plane,
$f_E(e^{2\pi iz})=\eta(z)^2\eta(11z)^2$, with~$\eta$ the Dedekind eta-function. The modular form $\eta(z)^2\eta(11z)^2$ is the unique normalized cuspform of weight $2$ and level $\Gamma_0(11)$, corresponding to the modular elliptic curve $X_0(11)$.
It also corresponds to the modular elliptic curve $X_1(11)$, given by
\begin{equation}
\label{ellipticequation}    
E:~y^2-y=x^3-x^2.
\end{equation}
According to Mazur \cite[p.\,198]{error}: ``This is an elliptic curve that is something of a showcase for number theory in that
it has been extensively studied—much is known about it—and yet it continues to
repay study, for, as with all other elliptic curves, its deeper features have yet to be
understood.'' Gelbart \cite[Examples 2.1, 2.2, 5.1]{Gelbart} and Joyner \cite[pp.\,40--42]{Joyner} also discuss $E$ and its associated $L$-series.

The coefficients $a_n$ bear some connection to the Ramanujan tau-function 
$\tau(n)$. Indeed, comparison of \eqref{cndefi} with
\eqref{productexpansion}
shows that $a_n\equiv \tau(n)\ \pmodi{11}$.
Cowles \cite{Cowles} used 
this and other known facts to compute $a_p$ from known values of $\tau(p)$ for 
primes $p\le 757$. 
Trotter, as an appendix to a paper
of Shimura \cite{shimura}, gave a table of $a_p$ for all primes  $p\le 2000$. These
days it is trivial to generate many of the $a_n$ (see
\url{https://oeis.org/A006571}).

The Fourier coefficients $a_p$ figure prominently in a finite number of reciprocity laws in non-solvable extensions formulated by 
Shimura \cite{shimura}. His paper is seen as a milestone in the early development of non-abelian class field theory.

We focus on congruences for $a_n$, but we note that also a lot of work on Fourier coefficients assuming some prescribed value has been done. For 
our example the Lang and Trotter~\cite{LT} conjecture
implies that there exists a constant 
$c\ge 0$ such that
$$\#\{n\le x:a_n=a\}\sim c\frac{\sqrt{x}}{\log x}.$$
A large number of generalizations and related results is possible, see Serre \cite{serreIHES}.

\subsection{Congruences for the Fourier coefficients of \texorpdfstring{$f_E$}{f\_E}
}
The weight-2 level-$11$ cuspform~$f_E$ with $q$-expansion $f_E = \sum_n a_n q^n$ whose Fourier coefficients are the $a_n$ is Eisenstein modulo $5$. This follows, for example, by Mazur's work \cite{mazur}, which guarantees a cuspform in prime level $N$ and weight $2$ congruent to the Eisenstein series {$$E_{2, N} := \frac{1}{24}\big(E_2(z) - N E_2(Nz)\big)$$}  modulo any prime dividing the numerator of $\frac{N-1}{12}$ combined with the fact that $S_2\big(\Gamma_0(11)\big)$ is spanned by $f_E$. In particular, this implies the mod-$5$
congruence
\begin{equation}
\label{modulo5}    
a_n\equiv \sigma_1\Big(\frac{n}{11^e}\Big)\,\pmod{5}\,, \quad \text{ with } 11^e\mid \mid n.
\end{equation}
See, for example, Chowla \cite{C}. In fact, using the methods of \cite[\S 3.2]{serrecong} one can easily show that $2$ and $5$ are the only ``exceptional" primes for $f$ in the sense of Serre \textit{loc.~cit.}. Using that the
elliptic curve $E$ from \eqref{ellipticequation} satisfies $E[5]\cong \mathbb Z/5\mathbb Z$ \cite[p.\,55]{LT},
it is easy to prove~\eqref{modulo5} if $n$ is a prime \cite[Lemma 9.1]{HIS}.

In the next two sections we consider our \cref{problem} for $f = f_E$, first for $\ell = 5$ and then for $\ell = 2$.

\section{Divisibility of coefficients of \texorpdfstring{$X_0(11)$}{X\_0(11)} by 5}
\label{five}

\subsection{Key computation of \texorpdfstring{\( \gamma_S \)}{gamma\_S} and \texorpdfstring{\( C_S \)}{C\_S}} Set 
$S:=S_{f_E,5}=\{n:5\nmid a_n\}$ and~\mbox{$F(s):=F_S(s)$}.
We are interested in computing
$\gamma_{S}$. 
We will use congruence \eqref{modulo5} to determine~$F(s)$ and subsequently solve
the \cref{problem}.
Clearly~$p$ is in
$S$ if and only if 
\mbox{$p\not\equiv 4 \pmodi{5}$}.
Thus, by the prime number theorem for arithmetic progressions, 
\eqref{primecondition} is satisfied with\linebreak\mbox{$\delta=3/4$}.
The set $S$ is multiplicative and hence
$F(s)$ has an Euler product. For a prime \mbox{$p\equiv 4 \pmodi{5}$} we 
observe that $5\nmid c_{p^{\alpha}}$ if and only if
$\alpha$ is even. This leads to a local factor
$$1+p^{-2s}+p^{-4s}+\cdots =(1-p^{-2s})^{-1}.$$
By a slightly more extended argument we see that for $p\ne 5,11$ the local factor equals
$$\frac{(1-p^{-(\nu_p-1)s})}{(1-p^{-s})(1-p^{-\nu_ps})},$$
where $\nu_p=5$ if $p\equiv 1 \pmodi{5}$ and $\nu_p$ is the multiplicative order of $p$ modulo $5$ otherwise.
We conclude that 
$$F(s)=(1-5^{-s})^{-1} (1-11^{-s})^{-1}{\widetilde F}(s),$$
with 
$$
{\widetilde F}(s)=\prod_{\substack{p\equiv 1 \pmodi{5} \\ p>11}}\frac{1-p^{-4s}}{(1-p^{-s})(1-p^{-5s})}
\prod_{p\equiv \pm 2 \pmodi{5}}\frac{1-p^{-3s}}{(1-p^{-s})(1-p^{-4s})}
\prod_{p\equiv 4 \pmodi{5}}\frac{1}{1-p^{-2s}},$$
where the local factors $(1-5^{-s})^{-1}$ and $(1-11^{-s})^{-1}$
come from the fact that $5\nmid c_{5^{\alpha}\cdot 11^{\beta}}$ for 
arbitrary~$\alpha,\beta\ge 0$.
The plan is now to relate $F(s)$ to Dirichlet $L$-series.
Let $\chi_c$ be the character of~$(\mathbb Z/5\mathbb Z)^*$ that is determined by
$\chi_c({\bar 2})=i$ and $\chi_5$ be the character that is determined by
$\chi_5({\bar 2})=-1$.
By comparison of local factors we
conclude that 
\begin{equation}
\label{basisvergelijking}
F(s)^4=(1-5^{-s})^{-1}\,H(s)\,\zeta(s)^3\frac{L(s,\chi_c)L(s,{\bar \chi_c})}{L(s,\chi_5)},\qquad \text{where}
\end{equation}
$$
H(s) = \prod_{\substack{p\equiv 1 \pmodi{5} \\ p>11}}\left(\frac{1-p^{-4s}}{1-p^{-5s}}\right)^4
\prod_{p\equiv \pm 2 \pmodi{5}}{\frac{(1-p^{-3s})^4}{(1-p^{-2s})^2(1-p^{-4s})^3}}
\prod_{p\equiv 4\pmodi{5}}\frac{1}{(1-p^{-2s})^{2}}.$$
On invoking \eqref{residu} with $\delta=3/4$, we obtain
$$C_{S}=\frac{1}{\Gamma({\frac{3}{4}})}\left(
\frac{5\,H(1)\,L(1,\chi_c)\,L(1,{\bar \chi_c})}{4L(1,\chi_5)}\right)^{1/4}.$$
In the book of Edwards 
\cite[p.\,388]{edwards} the $L$-values above are given (for an excellent discussion of how
to compute these values see \cite[6.5]{edwards}),
resulting in 
$$L(1,\chi_c)\,L(1,{\bar \chi_c})=\frac{2\pi^2}{25},\quad L(1,\chi_5)=\frac{2}{\sqrt{5}}\log\bigg(\frac{1+\sqrt{5}}{2}\bigg).$$
We thus obtain the explicit expression
\begin{equation}
\label{maincoefficient}    
C_{S}={\frac{1}{\Gamma({\frac{3}{4}})}}\left(\frac{H(1)\,\pi^2}{4\sqrt{5}\log(\frac{1+\sqrt{5}}{2})}\right)^{1/4}.
\end{equation}

\begin{rk}\label{rk:est:sumlog1mk}
With the naive bound \( 0 < -\log(1-n^{-k}) < 2n^{-k} \), holding for \( n \geq 2 \) one has, by the integral test, for \( k > 1 \), that the tail of the following series satisfies 
\begin{equation}\label{eqn:log1m1ontail}
    0 > \sum_{n=x+1}^\infty \log\Big(1-\frac{1}{n^{k}}\Big) > -\frac{2}{(k-1) x^{k-1}} \,.
\end{equation}
One can thus estimate the sum \( \sum_{n=2}^\infty \log(1-n^{-k}) \), also when restricting to primes satisfying certain conditions, by explicitly summing the terms \( n \leq x \), and using the  tail \eqref{eqn:log1m1ontail} to bound the possible error.
\end{rk}

From \cref{rk:est:sumlog1mk}, one obtains via interval arithmetic, and summing the first~\( x = 10^8 \) terms, that
\begin{align*}
    \log H(1) &\in (0.42434580\ldots, 0.42434588\ldots) \,,  \\
     \text{ so $H(1)$} &\in (1.52859010\ldots, 1.52859022\ldots) 
  \end{align*}
implying $C_{S}\in (1.11658960\ldots, 1.11658962\ldots)$ and
so
\[
    C_{S} = 1.1165896\ldots \,,
\]
correct to 7 digits.

\indent Logarithmic 
differentation of \eqref{basisvergelijking} in 
combination with \eqref{EKf}  and $\gamma_{\mathbb N}=\gamma$ gives
$$4\gamma_{S}=3\gamma
+2\Re\left(\frac{L'(1,{\chi_c})}{L(1,{\chi_c})}\right)
-\frac{L'(1,\chi_5)}{L(1,\chi_5)}
-\frac{\log 5}{4} -A_{\pm 1}-A_{\pm 2}.$$
where 
$$A_{\pm 1}=4\sum_{p\equiv 4 \pmodi{5}}{\frac{ \log  p}{p^2-1}}-\sum_{\substack{p\equiv 1 \pmodi{5} \\ p>11}}\left({\frac{16\log p}{p^4-1}}-{\frac{20\log p}{p^5-1}}\right)
\text{ and } $$
$$A_{\pm 2}=\sum_{p\equiv \pm 2 \pmodi{5}}\left({\frac{4\log p}{p^2-1}}-{\frac{12\log p}{p^3-1}}
+{\frac{12\log p}{p^4-1}}\right).$$

\begin{rk}\label{rk:est:sumlogpoverpkm1}
The prime sums of the form \( \sum_{\text{$p$ prime}} \frac{\log{p}}{p^k-1} \) typically appearing in this story can be naively estimated
by computing all terms up to some large value $x$ and estimating the tail by
\begin{equation}
\label{trivialtail}    
0 < \sum_{p>x}\frac{\log p}{p^k-1}\le \frac{x}{x^k-1}\Bigl(-0.98+1.017\frac{k}{k-1}\Bigr),\quad k\in \mathbb R_{>1}\text{~and~}x\ge 7481,
\end{equation}
which follows easily on using the estimate $0.98x \le \sum_{p\le x}\log p \le 
1.017x$ for $x \ge 7481$ due to Rosser and Schoenfeld \cite{RS}. 
\end{rk}

{}From \cref{rk:est:sumlogpoverpkm1}, with \( x = 10^8 \), one obtains
\begin{align*}
    A_{\pm1} \in (0.06745302\ldots, 
0.06745310\ldots), \\
    A_{\pm2} \in (0.76345581\ldots, 
0.76345589\ldots).
\end{align*}

\begin{rk}\label{rk:est:sumlogpoverpkm2}
For prime sums of the form $$\sum_{p\equiv a \pmodi{d}}{\frac{ \log  p}{p^r-1}},$$
especially with small $d$, one can do considerably better than by the all-purpose approach from  \cref{rk:est:sumlogpoverpkm1}, 
see \cite{LaMoree} and \cref{rk:arith} below.
This allows one, for example, to compute~$A_{\pm 1}$ and $A_{\pm 2}$ with hunderds of decimals of accuracy.
\end{rk}

Finally one can readily evaluate \( L'/L(1, \chi_c) \) and \( L'/L(1, \chi_5) \) to high accuracy, for example with \cite[Chp.\,9]{BC}, 
\cite[Sec.\,3]{Ale} and
\cite[Sec.\,4]{LR}.  We calculate these values directly with  \texttt{Pari/GP}, giving
\begin{align*}
\frac{L'(1, \chi_5)}{L(1, \chi_5)} &= 0.8276794767155048879910469\ldots \\
\Re \Big( \frac{L'(1, \chi_c)}{L(1, \chi_c)} \Big) &= 0.1578645354481835065490848\ldots \,.
\end{align*}

Assembling these ingredients, we find
\(
    \gamma_S \in (-0.00339297, 
-0.00339293) 
\), hence 
\[
\gamma_S = -0.0033929\ldots \,,
\]
accurate to 7 decimal places.

As $\gamma_{S}$ is non-zero, it
		follows that \eqref{boldclaim} for 
		$S=S_{f_E,5}$ is \emph{false}. However, 
		asymptotically $S(x)$ 
		is better approximated 
		by the integral $C_{S}\int_2^x\log^{-1/4}t\,dt$, than by 
		$C_{S}x\log^{-1/4}x$.

    We conclude this section by discussing various alternative approaches to derive \( \gamma_S \) and~\( C_S \) via previously studied examples, as well as how to express them in terms of other \( \zeta \)-functions, and derive much more accurate values.

    \subsection{Relation to Ramanujan \texorpdfstring{$\tau$}{tau}}
		These expressions for $C_{S}$ and $\gamma_{S}$ can also be derived in a different way. 
		Ramanujan observed that $\tau(n)\equiv n\sigma_1(n) \pmodi{5} $. It follows from \eqref{modulo5} that if $(n,55)=1$, then 
		$5\nmid \tau(n)$ 
		if and only if $5\nmid a_n$. This means that we can analyze the $5$ non-divisibility of 
		$a_n(f_E)$ by making some small variations in the existing analysis of the $5$ non-divisibility of $\tau(n)$ as given in Moree \cite{MRama}. Namely, we have to only modify the local factors at $5$ and $11$. This gives
		$$F(s)=\frac{1}{1-5^{-s}}\cdot \frac{1-11^{-5s}}{1-11^{-4s}}\cdot \sum_{n\in S'}\frac{1}{n^s},\quad\text{where~}
		S':=\{n:5\nmid n\sigma_1(n)\}.$$
		It follows that
		$$C_{S}=\frac{5}{4}\cdot \frac{1-11^{-5}}{1-11^{-4}}\cdot C_{S'},$$
		which together with the expression for $C_{S'}$ given in \cite{MRama} results again 
		in \eqref{maincoefficient}.  
		
		In the notation of Ciolan et al.\,\cite{CLM} the Euler-Kronecker constant associated to $S$ is
		$\gamma'_{1,5}$, with $\gamma'_{1,5}=0.399547\ldots$ (giving a few more accurate decimals than obtained in \cite{MRama}).
		We conclude again that 
		\[
		\gamma_{S}=-\frac{\log 5}4+\frac{5\log 11}{11^5-1}-
		\frac{4\log 11}{11^4-1}+\gamma'_{1,5}=-0.00339\ldots. \qedhere 
		\]
		
	\subsection{Evaluation via Dedekind \texorpdfstring{$\zeta$}{zeta}-functions}
		Put $K_1=\mathbb Q(\zeta_5)$ and $K_2=\mathbb Q(\sqrt{5})$.
		We have \( \zeta_{K_2}(s) = \zeta(s) L(s, \chi_5)  \) and \( \zeta_{K_1}(s) = \zeta_{K_2}(s) L(s, \chi_c) L(s,\bar{ \chi}_c) \). Using this we can rewrite~\eqref{basisvergelijking} as
		\begin{equation}
		\label{basisvergelijkingD}
		F(s)^4=(1-5^{-s})^{-1}\,H(s)\,\zeta(s)^4\frac{\zeta_{K_1}(s)}{\zeta_{K_2}(s)^2},
		\end{equation}
		in line with  \cref{simpleanna}.
		Logarithmic differentiation then leads to
		$$4\gamma_{S}=4\gamma
		+\gamma_{K_1}-2\gamma_{K_2}
		-(\log 5)/4 -A_{\pm 1}-A_{\pm 2}.$$
		Languasco \cite[Table~2, p.13]{Ale} calculated
		\begin{align*}
		\gamma_{K_1} &= 1.72062421251340476169572878865\dotsc \,, \\
		\gamma_{K_2} &= 1.40489514161703774859755907976\dotsc \,,
		\end{align*}
		(the last digit may be rounded by Pari/GP).

	\subsection{Evaluating sums of primes in arithmetic progressions}
	\label{rk:arith}
		Following the ideas of \cite[Remark 1]{LaMoree}, one can evaluate the sums \( A_{\pm1}, A_{\pm2} \) and product \( C_S \) (or more generally any sum or product taken over primes in a union of congruence classes) to high accuracy.  We briefly sketch the main idea, and give the resulting evaluations. 
		
		From the Euler product of the Dirichlet $L$-function of character $\chi$, we have that
		\[
		\log L(s, \chi) = \sum_{n \geq 1} \frac{P_{\chi^n}(n s)}{n} \,,	\text{ where }  P_{\chi}(s) = \sum_{\text{$p$ prime}} \frac{\chi(p)}{p^{s}} \,.
		\]
    By M\"obius inversion, it follows from this that
		\[
		P_\chi(s) 
		 = \sum_{n \geq 1}	\frac{\mu(n)}{n} \log L(\chi^n, n s).
		\]
		The tail of the series may be easily estimated, including after differentiation with respect to~$s$.   With established procedures for evaluating \( L(\chi,s) \) and \( L^{(}{'}{}^{)}(\chi,s) \), cf.\,Languasco 
		\mbox{\cite[Sec.\,8]{Aleunified}}\footnote{The built-in algorithm implemented in \texttt{Pari/GP} \cite[Chp.\,9]{BC} is substantially slower, but would suffice. For a comparison see \cite[Sec.\,10]{Aleunified}.}, one can accurately evaluate the leading part of the sum, and hence \( P_\chi(s) \) and~\( P_\chi'(s) \).
		
		We hence can obtain accurate values for the following sums
		\begin{align*}
			P_\chi(s) = \sum_{\text{$p$ prime}} \frac{\chi(p)}{p^s} \,,
			\quad\quad
			P_\chi'(s) = \sum_{\text{$p$ prime}} \frac{\chi(p) \cdot \log{p}}{p^s} \,.
		\end{align*}
		Since the indicator function for an arithmetic progression can be expressed as a sum of characters, we can therefore accurately evaluate such sums when the primes lie in such a union of arithmetic progressions.  We can then Taylor expand a general rational function summand into the first form, and \( \log{p} \) times a general rational function summand into the second form.  By appropriately estimating the error, one can obtain certified results accurate to the desired precision.
		
		We obtain (remembering to remove the extra $ p = 11 $ contribution)
		\begin{align*}
		\log H(1) = 0.424345810469665379222671735635\ldots \\
		H(1) = 1.528590104858545837177881690521\ldots \\
		C_S = 1.116589605788653881432691480834\ldots \\
		A_{\pm 1} =  0.067453075884830773681197940988\ldots \\
		A_{\pm 2} = 0.763455872211726828059807846118\ldots 
		\end{align*}
		Hence
		\begin{equation}
		\label{gammatrivial}    
		\gamma_{S} = -0.003392959329905497116134157738\ldots \,,
		\end{equation}
		with many more accurate digits.

	\subsection{Arithmetic progression zeta functions}
		In case the set $S$ consists, with finitely many exceptions, precisely
		of all primes in a finite union of arithmetic progressions, we can express the generating
		series in terms of the following type of zeta functions 
		\begin{equation}
		\label{Fda-def}
		\zeta_{\,d,a}(s):=\prod_{p\equiv a \pmodi{d}}\frac{1}{1-p^{-s}},
		\end{equation}
		with $d$ and $a$ coprime. The associated Euler-Kronecker constants $\gamma(d,a)$ were introduced 
		and studied by Languasco and the third author \cite{LaMoree}. Using \eqref{Fda-def} we see that
		$$F(s)=\frac{\zeta(s)}{\zeta_{5,4}(s)}\prod_{\substack{p\equiv 1 \pmodi{5} \\ p>11}}\frac{1-p^{-4s}}{1-p^{-5s}}
		\prod_{p\equiv \pm 2 \pmodi{5}}\frac{1-p^{-3s}}{1-p^{-4s}}
		\prod_{p\equiv 4 \pmodi{5}}\frac{1}{1-p^{-2s}},$$
leading to
		$$\gamma_{S}=\gamma-\gamma(5,4)+B,$$
		where 
		$$B=\sum_{\substack{p\equiv 1 \pmodi{5} \\ p>11}}\left({\frac{4\log p}{p^4-1}}-{\frac{5\log p}{p^5-1}}\right)+\sum_{p\equiv \pm 2\pmodi{5}}\left({\frac{ 3\log  p}{p^3-1}}-{\frac{ 4\log  p}{p^4-1}}\right)-\sum_{p\equiv 4 \pmodi{5}}{\frac{ 2\log  p}{p^2-1}}.$$
		By \cite[Table 1]{LaMoree} we have
		$\gamma(5,4)=0.75093 25558 29083 25836 43416 62578 48050 \ldots$.
		The constant $B$ can be 
		computed with high-precision as 
		explained in \cite[Remark~1]{LaMoree}, and briefly recapitulated in \cref{rk:arith} above.
		In particular, using \texttt{gammaS\_ver} from Languasco\footnote{See \scriptsize \url{https://www.dei.unipd.it/~languasco/EulerconstAP-site/gp/sumprimes_derlog_AP.gp}.}, one obtains
		$$B =0.17032 39315 97644 90064 16954 14757 47429\dotsc.$$ 
Putting everything together one sees that \eqref{gammatrivial} is confirmed.

\section{The parity of the Fourier coefficients of \texorpdfstring{$X_0(11)$}{X\_0(11)}}
\label{parity}
\label{sec:parity}
We continue our notation: 
$$f_E=\eta(z)^2\eta(11z)^2=q - 2 q^{2} -  q^{3} + 2 q^{4} +  q^{5} + 2 q^{6} - 2 q^{7} - 2 q^{9} - 2 q^{10} +  q^{11} + O(q^{12})$$
is the unique normalized cuspform of weight $2$ and level $\Gamma_0(11)$, corresponding to the modular elliptic curve $X_0(11)$. We set $\ell = 2$ and 
$S=S_{f_E,2}$.
\subsection{Parity characterizations}
In this section we recall some results from the literature showing
that the parity of $a_n(f_E)$ carries interesting arithmetic information.

Let $\#E(\mathbb F_p)$ be the number of $\mathbb F_p$-rational points of an elliptic curve $E$ modulo $p$. Denote by $N_p(f)$ the number of distinct roots modulo $p$ of a polynomial $f\in \mathbb Z[x]$ and let~$r_{12}(n)$ be the number of representations of
$n$ as a sum of twelve squares. The Tribonacci sequence~$(T_n)_{n\ge 1}$ is defined by $T_1=1, T_2=1, T_3=2$, and 
by the third 
order linear recurrence~$T_n=T_{n-1}+T_{n-2}+T_{n-3}$ for every $n\ge 4.$
\begin{Prop}
\label{bunchofequivalences}
Suppose that 
$p\equiv 1,3,4,5,9 \pmodi{11}$ is a prime number. Then the following are equivalent:
\begin{enumerate}
    \item \label{(1)}$a_p$ is even;
    \item\label{(2)} $\#E(\mathbb F_p)$ is even with $E$ being the elliptic curve $E:~y^2-y=x^3-x^2$;\footnote{Mazur \cite{error} discusses for a general audience the distribution of the numbers 
$\#E(\mathbb F_p)$ in the context of the Sato-Tate conjecture. His article is an excellent first introduction to $f_E$ and its Fourier coefficients.}    
    \item \label{(3)} $p=X^2+11Y^2$ \cite{CC};
     \item \label{(4)} $N_p(4X^3-4X^2+1)=3$ \cite{CC};
     \item \label{(5)} $r_{12}(p)\equiv 4 \pmodi{11}$ \cite[(5.4)]{Ahlgren};
     \item \label{(6)} $p$ divides $T_{p-1}$ \cite{EH}.
\end{enumerate}
\end{Prop}
If $E$ is an elliptic curve over $\mathbb Q$ and $p$ is a good prime for $E$,
then $\#E(\mathbb F_p)=1+p-c_p$, with $c_p$ the trace of Frobenius.
In our case $c_p=a_p$ and this proves the equivalence of~\ref{(1)} and~\ref{(2)}.
Chowla and Cowles \cite{CC} established the equivalence of \ref{(3)} and \ref{(4)} with \ref{(1)}. 
Ahlgren \cite{Ahlgren} showed  the equivalence of \ref{(5)} with \ref{(1)} (for all primes $p\ne 11$).
Evink and Helminck \cite{EH}) proved the equivalence of \ref{(1)}, \ref{(3)}, and \ref{(6)} (and were apparently not aware of \cite{CC}).

Equivalences between  \ref{(1)} (Fourier coefficients) and \ref{(4)} (number of roots of a polynomial modulo $p$) hold in greater 
generality, see Khare et al.\,\cite{KLW} and Serre \cite{serrejordan}. The same holds for equivalences  between \ref{(1)} 
and \ref{(6)} (linear recurrences),
see Moree and Noubissie \cite{MN} and Rosen et al.~\cite{RTTY}.

Ito \cite{Ito} reproved the equivalence of \ref{(2)} and \ref{(3)} more conceptually and gave similar examples. His starting point is an elliptic curve
$E$ given in Weierstrass normal form \mbox{$Y^2=X^3+AX+B$}, with $A$ and $B$ integers.
Then the 2-division field $\mathbb Q(E[2])$ is the splitting field of the equation
$X^3+AX+B$, and so is Galois over $\mathbb Q$. It has 
$\mathbb Q(\sqrt{\Delta})$ with $\Delta=-2^4(4A^3+27B^2)$ as a subfield.
Ito relates the parity of $\#E(\mathbb F_p)$ to the splitting behavior of $p$ in 
$\mathbb Q(E[2])$ in case  $\mathbb Q(E[2])$ is cyclic of degree 3 over $\mathbb Q$.
In this way he reproves the result of Chowla and Cowles and gives some further similar examples, where now we restrict to the primes splitting in 
$\mathbb Q(\sqrt{\Delta})$ rather than those splitting in
$\mathbb Q(\sqrt{-11})$ as in \Cref{bunchofequivalences}. 
For~$X_0(11)$ we have Weierstrass form $Y^2=X^3-13392X-1080432,$ leading to
a field with defining polynomial $$X^6-80352X^4+1614110976X^2+21910810749696,$$ which 
turns out to be isomorphic to the field $L$ defined in the next section.

\subsection{The Dirichlet series \texorpdfstring{$F(s)$}{F(s)}}
Since the $a_n(f_E)$ are multiplicative, so is the characteristic function
of $S$, and hence $F(s)$ has an Euler product expansion.
Because $c$ modulo~$2$ is a dihedral modular eigenform with field of determination
$L := {}$ 2-ray class field of~$K:= \QQ(\sqrt{-11})$ (cf.\,Milne \cite[Chp.\,5]{Milne}), we can determine the individual Euler factors $F_p(s)$ for the primes unramified in $L/\QQ$ by considering their conjugacy class in $\Gal(L/\QQ)$ under the Artin map $p \mapsto \Frob_p$. Here $\Frob_p \subseteq \Gal(L/\QQ)$ is the conjugacy class of elements~$\sigma$ in $\Gal(L/\QQ)$ so that for some prime $\mathfrak P$ of $L$ above $p$ we have $\sigma(\alpha) \equiv \alpha^{p}$ modulo $\mathfrak P$ for every $\alpha \in \OO_L$, the ring of integers of $L$. The field $L$
\cite[{\href{https://www.lmfdb.org/NumberField/6.0.21296.1}{\texttt{NumberField/6.0.21296.1}}}]{LMFDB} 
is an extension of $K$ of degree 3.
The extension \( L/\mathbb{Q} \) is \( {\mathfrak S}_3 \), with defining polynomial \[ X^6 - X^5 + 2X^4 - 3X^3 + 2X^2 - X + 1. \]

In addition to the ramified primes $2$ and $11$, we consider the following three sets of primes corresponding to the three conjugacy classes in $\Gal(L/\QQ) \simeq {\mathfrak S}_3$:

	\begin{center}
		\renewcommand{\arraystretch}{1.5}
		\begin{tabular}{c|c|c|c|c}
		Prime set & Order of $\Frob_p$ & $p$ in $K/\QQ$ & $\pp$ in $L/K$ & Sample primes \\ \hline\hline
		$\CCC_1$ & 1 & split  & split  & 47, 53, 103, 163, 199, \ldots 
		\\
		$\CCC_2$ & 2 & inert & split & 7, 13, 17, 19, 29, 41 \ldots 
		\\
		$\CCC_3$ & 3 & 
		split 
		& inert & 3, 5, 23, 31, 37, 59, 67, \ldots 
	\end{tabular}
	\end{center}
By the Chebotarev density theorem these sets of primes have natural densities corresponding to the densities of number of elements of order 1, 2, respectively 3 in $\SS_3$. These densities are of course $1/6$, $1/2$, and $1/3$. 

We now compute $F_p(s)$ for each behavior in $\Gal(L/\QQ)$ separately using the Hecke recurrence
$$a_{p^k} = a_p a_{p^{k-1}} - p a_{p^{k-2}}$$
that holds for all $k \geq 2$ for primes $p \neq 11$, along with $a_{11^k} = a_{11} a_{11^{k-1}}$ for all $k \geq 1$. 

Write $\bar a_n$ for $a_n$ modulo $2$, so that $n \in S$ if and only if $\bar a_n = 1$. Let $(\rho, V)$\label{rho} be the two-dimensional irreducible representation of $\Gal(L/\QQ) \simeq {\mathfrak S}_3$, the \emph{standard representation}. Since $(\rho, V)$ is defined over $\ZZ$ we can also reduce it modulo $2$ to obtain $(\rhobar, \Vbar)$, a faithful representation 
$$\rhobar: \Gal(L/\QQ) = {\mathfrak S}_3 \to \GL(\Vbar) = \GL_2(\FF_2).$$
Since $\bar c$ is a mod-$2$ eigenform with mod-$2$ Galois representation $\rhobar$, 
we have
$$\bar a_p = \tr(\Frob_p \mid \Vbar)$$
for all primes $p \neq 2, 11$. The sequence $\{\bar a_{p^k}\}_{k \geq 0}$, as a linear recurrence sequence of order~$\leq 2$ in a finite field of order $2$, is periodic of period at most 4. We therefore easily compute, keeping in mind that $\bar a_1 = 1$,
\begin{center}\begin{tabular}{|c||c|c|c|c|c|}
\hline
    $p$    & $\Frob_p$ &  $\bar a_p$  & $\bar a_{p^k} = $ & $\bar a_1, \bar a_p, \ldots \bar a_{p^4}$
     & $F_p(s)$\\
     \hline\hline
    $p \in \CCC_1$ & id & 
    $0$ &
    \multirow{3}{*}[-1.9em]{ $\bar a_p \bar a_{p^{k-1}} + \bar a_{p^{k-2}}$} &
    1, 0, 1, 0, 1 & 
    $\displaystyle \frac{1}{1 - p^{-2s}} \phantom{\Bigg)}$ \\
    \cline{1-3} \cline{5-6}
    %
    $p \in \CCC_2$ & $2$-cycle & 
    $0$ &
    &
    1, 0, 1, 0, 1 & 
    $\displaystyle \frac{1}{1 - p^{-2s}} \phantom{\Bigg)}$\\
    \cline{1-3} \cline{5-6}
    $p \in \CCC_3$ & $3$-cycle &
    $1$ &
    &
    1, 1, 0, 1, 1 & 
    $\displaystyle 
    \frac{1 + p^{-s}}{1 - p^{-3s}} \phantom{\Bigg)}$\\
    \hline    
    2 & --- & 0 &
    \multirow{2}{*}[-1em]
    {$\bar a_p \bar a_{p^{k-1}}$}
    & 1, 0, 0, 0, 0 & 
    1\\
   \cline{1-3} \cline{5-6}
    11 & --- &1 & 
    & 1, 1, 1, 1, 1 & 
    $\displaystyle \frac{1}{1-11^{-s}} \phantom{\Bigg)}$\\
    \hline
\end{tabular}\end{center}
leading to 
\begin{equation}
\label{eqn:fdef}  
F(s) = \frac{1}{1 - 11^{-s}} \prod_{p \in \CCC_1 \,\cup \,\CCC_2} \frac{1}{1 - p^{-2s}} \prod_{p \in \CCC_3} \frac{1 + p^{-s}}{1 - p^{-3s}}.
\end{equation}
The idea is to express $F(s)$ in terms of known $L$-series. As the sets
$\CCC_1$ and $\CCC_3$ are non-abelian, we expect we cannot only do with Dirichlet 
$L$-series and will need Artin $L$-series.  In this case, it turns out that invoking only \( L(\rho, s) \) for $\rho$ as on p.\,\pageref{rho} above is sufficient to express~\( F(s) \) in the form we want, given in \eqref{eqn:fv2} below.

\subsection{The Artin \texorpdfstring{$L$}{L}-function of \texorpdfstring{$(\rho, V)$}{(rho, V)}}\label{sec:parity:artin}

The following table shows the computation of $L_p(\rho,s) \coloneqq  L_{\Q,p}(\rho, s)$ as in  \eqref{Artinlocalfactor}. 

\begin{center}\begin{tabular}{|c||c|c|c|c|}
\hline
    $p$    & 
    $I_p$ &  
    $V^{I_p}$ & 
    $\big(\Frob_p \mid V^{I_p}\big)$ & 
    $L_p(\rho,s)$ 
    \\
    \hline\hline
    $p \in \CCC_1$ & \multirow{3}{*}[-2em]{\{1\}} & 
    \multirow{3}{*}[-2em]{$V$} &
    $\begin{psmallmatrix} 1 & 0 \\ 0 & 1 \end{psmallmatrix}$ &
    $\displaystyle \frac{1}{(1 - p^{-s})^2} \phantom{\Bigg)}$\\
    \cline{1-1} \cline{4-5}
    $p \in \CCC_2$ & 
    &
    &
    $\begin{psmallmatrix} 0 & 1 \\ 1 & 0 \end{psmallmatrix}$ & 
    $\displaystyle \frac{1}{1 - p^{-2s}} \phantom{\Bigg)}$\\
    \cline{1-1} \cline{4-5}
    $p \in \CCC_3$ & 
    &
    &
    $\begin{psmallmatrix} e^{\smash{2\pi i/3}} & 0 \\[0.2ex] 0 & e^{\smash{-2\pi i/3}} \end{psmallmatrix}$ & 
    $\displaystyle \frac{1}{1  + p^{-s} + p^{-2s}} \phantom{\Bigg)}$\\
    \hline    
    2 &
    $3$-cycles& 
    $\{0\}$&
    --- & 
    1\\
\hline    %
    11 &
    $2$-cycles& 
    line&
    $\begin{psmallmatrix} 1 \end{psmallmatrix}$ & 
    $\displaystyle \frac{1}{1 - 11^{-s}} \phantom{\Bigg)}$\\
\hline    
\end{tabular}\end{center}

For the ramified primes we fill in the table by computing the factorizations in $\OO_L$: the prime $(2)$ is the cube of a prime of $\OO_L$, so that $I_2$ is a group of order $3$; the prime $(11)$ is the square of the product of three primes of $\OO_L$, so that $I_{11}$ is any of the three $2$-cycles of~$\Gal(L/\QQ)$.

It follows that
\begin{equation*}
 L(\rho, s) = \frac{1}{1 - 11^{-s}} \prod_{p \in \CCC_1} \frac{1}{(1 - p^{-s})^2} \prod_{p \in \CCC_2} \frac{1}{1 - p^{-2s}}
 \prod_{p \in \CCC_3} \frac{1 - p^{-s}}{1 - p^{-3s}}.
\end{equation*}

Comparing the Euler product \eqref{eqn:fdef}  for $F(s)$ with that for $L(\rho,s)$ derived in the previous section, gives
\begin{equation}\label{eqn:fv2}
	F(s)^3 = (1 - 2^{-2s})^4 (1 + 11^{-s})^3\frac{\zeta_K(s) \zeta(2s)^3}{L(\rho, s)} \prod_{p \in \CCC_3} \frac{(1 - p^{-2s})^6}{(1 - p^{-3s})^4}\,,
\end{equation}
on observing that
\begin{align*}
	\zeta_K(s) &=  \frac{1}{1 - 2^{-2s}} \cdot \frac{1}{1 - 11^{-s}} \cdot \prod_{p \in \CCC_1 \cup \CCC_3} \frac{1}{(1 - p^{-s})^2} \prod_{p \in \CCC_2} \frac{1}{1 - p^{-2s}};\\
	\zeta(2s)^3 &= \frac{1}{(1 - 2^{-2s})^3} \cdot \frac{1}{(1 - 11^{-2s})^3}\prod_{p \in \CCC_1 \cup \CCC_2 \cup \CCC_3} \frac{1}{(1 - p^{-2s})^3}. 
\end{align*}

\begin{rk}
The more general method illustrated in \cref{Haberland59} will give us 
    \[        F(s)^3 = (1 - 2^{-2s})^4 (1 + 11^{-s})^3\frac{L(\one, s) L(\sgn, s) \zeta(2s)^3}{L(\rho, s)} \prod_{p \in \CCC_3} \frac{(1 - p^{-2s})^6}{(1 - p^{-3s})^4}\,,
    \]
    which matches \eqref{eqn:fv2} (initially obtained in a more ad-hoc manner), as $L(\one,s) = \zeta(s)$ and  $\zeta_K(s) = L(\sgn,s) \zeta(s)$ (after one extracts a factor \( \zeta(2s)^3 \) from \( H(s) \)).
\end{rk}

\subsection{Reformulation in terms of Dedekind \texorpdfstring{$\zeta$}{zeta}-functions}

Usinsg Frobenius reciprocity, one can verify that \(
    \Ind_{\mathfrak{A}_3}^{\mathfrak{S}_3} \one = \one \oplus \sgn \),
where \( \sgn = \sgn_3 \) is the character on \( \mathfrak{S}_3 \).  Hence the regular representation satisfies
\[
    \Ind_{\{e\}}^{\mathfrak{S}_3} \one = \one \oplus \sgn \oplus \rho^{\oplus2} = \Ind_{\mathfrak{A}_3}^{\mathfrak{S}_3} \oplus \rho^{\oplus2} \,.
\]
In the extension \( \Gal(L/\QQ) \cong \mathfrak{S}_3 \), \( K = L^{\mathfrak{A}_3} \) and \( L = L^{\{e\}} \).  So by the behaviour  Artin $L$ functions under induced representations, and their connection with Dedekind \( \zeta \)-functions (see \cref{sec:Artin}), we obtain
\[
    \zeta_L(s) = \zeta_K(s)\, L_\Q(\rho, s)^2 \,.
\]

On squaring \eqref{eqn:fv2} and replacing $L(\rho,s)^2$ by $\zeta_L(s)/\zeta_K(s)$, we obtain
\begin{equation}\label{eqn:fvArtinfree}
	F(s)^6 = (1 - 2^{-2s})^8\cdot (1 + 11^{-s})^6\cdot \frac{\zeta_K(s)^3 \zeta(2s)^6}{\zeta_L(s)} \cdot \prod_{p \in \CCC_3} \frac{(1 - p^{-2s})^{12}}{(1 - p^{-3s})^8}.
\end{equation}

\subsection{Computation of the Euler-Kronecker constant}
Notice that \eqref{eqn:fvArtinfree} is of the form
$$F(s)^6=\frac{\zeta_K(s)^3}{\zeta_L(s)}H(s),$$
with
\[ 
H(s) \coloneqq (1 - 2^{-2s})^8 \cdot (1 + 11^{-s})^6\cdot \zeta(2s)^6 \cdot \prod_{p \in \mathcal{C}_2} \frac{(1 - p^{-2s})^{12}}{(1 - p^{-3s})^8}.
\] absolutely converging for $\Re(s)>1/2$. Taking the log-derivative and adding $2/(s-1)$ on both sides gives
$$6\frac{F'}{F}(s)+\frac{2}{s-1}=3\frac{\zeta_K'}{\zeta_K}(s)+\frac{3}{s-1}-\Big(\frac{\zeta_L'}{\zeta_L}(s)+\frac{1}{s-1}\Big)+\frac{H'}{H}(s).$$
Letting $s$ tend to $1$ from above along the real line and invoking \eqref{define}, then show
the existence of $\gamma_{S}$ and, moreover, we obtain
$$6\gamma_{S}=3\gamma_K-\gamma_L+\frac{H'}{H}(1).$$
Working out further the latter term involving $H$ 
involves only elementary calculus and is left
to the reader. We obtain 
$$6\gamma_{S}=\frac{16}{3}\log 2 - \frac{6}{12}\log 11+3\gamma_K-\gamma_L+
12\frac{\zeta'}{\zeta}(2)+24\sum_{\p \in \CCC_3}\Big(\frac{\log p}{p^2 - 1} - \frac{\log p}{p^3 - 1}
		\Big),$$
that is
$$\gamma_{S}=\frac{8}{9}\log 2 - \frac{\log 11}{12}+\frac{\gamma_K}{2}-\frac{\gamma_L}{6}+
2\frac{\zeta'}{\zeta}(2)+4\sum_{\p \in \CCC_3}\Big(\frac{\log p}{p^2 - 1} - \frac{\log p}{p^3 - 1}	\Big).$$

The Ihara bounds \eqref{eqn:ihara} with \( x = 10^8 \) give
\begin{align*}
    \gamma_K &\in (0.49294357\ldots,0.49305431\ldots) \,, \\
    \gamma_L &\in (2.00125903\ldots, 2.00155473\ldots) \,.
\end{align*}
Unconditionally, $\gamma_K$ can be obtained more accurately, see \cite{ibid}, namely
\[
    \gamma_K = 0.49299736060449193551282470608672156219\ldots
\]
Likewise, from the first \( x = 10^8 \) terms, and the tail from \eqref{trivialtail} in \cref{rk:est:sumlogpoverpkm1}, 
\begin{equation}\label{eqn:eg1:H}
  4 \sum_{p \in \CCC_3} \Big( \frac{\log{p}}{p^2 - 1} - \frac{\log{p}}{p^3 - 1} \Big) \in (0.67043751\ldots, 0.67043759\ldots) \,.
\end{equation}
This gives overall (using the accurate value for \( \gamma_K \) above), that
\[
    \gamma_{S} \in (-0.14027198\ldots, -0.14022265\ldots) \,,
\]
so \[\gamma_{S} = -0.1402\ldots, \] correct to four digits\footnote{The fourth digit does rely on the accurate version of \( \gamma_K \) obtained above.}.

\begin{rk}[Increased precision from heuristic \texttt{Pari/GP} calculation of \( \gamma_L \)]
One can apply the approach from \cref{app:higherprec} to calculate \( \gamma_L \) to higher precision using \texttt{Pari/GP}'s numerical $L$-function routines.  For the field \( L \) (which has defining polynomial \(  X^6 - X^5 + 2X^4 - 3X^3 + 2X^2 - X + 1 \))  one obtains
\[
     \gamma_L \overset{?}{=} 2.0014416775755487155542192\ldots
\]
From the accurate unconditional value of \( \gamma_K \), and the estimate in \eqref{eqn:eg1:H}, we hence obtain
\[
    \gamma_{S} \overset{?}{=} -0.1402(531)\ldots \,,
\]
with 3 extra \emph{uncertified} digits.
\end{rk}

\subsection{Computation of \texorpdfstring{\( C_{S} \)}{C\_S}}  Application of \eqref{residu}, with \( \delta = 1/3 \) to \eqref{eqn:fvArtinfree} gives
$$
    C_{S} = \frac{3^{4/3}}{4^{4/3}}\cdot \frac{2 \pi^2}{11 \Gamma(1/3)} \cdot \prod_{p\in\CCC_3} \frac{(1-p^{-2})^{2}}{{(1 - p^{-3})^{4/3}}} \cdot \frac{\big(\Res_{s=1} \zeta_K(s))^{1/2}}{\big( \Res_{s=1} \zeta_L(s) \big)^{1/6}}.
$$
From the analytic class number formula, see, e.g.\,\cite[Thm.\,61]{FT}, one can obtain
\begin{align*}
    \Res_{s=1} \zeta_K(s) &= \frac{\pi}{\sqrt{11}} \,,  \quad
    \Res_{s=1} \zeta_L(s) = \frac{\pi^3}{11\sqrt{11}} \Reg_L \,,
\end{align*}
and so
\[
    C_{S} = \frac{3^{4/3}}{4^{4/3}} \frac{2 \pi^2}{11 \Gamma(1/3)} \cdot \prod_{p\in\CCC_3} \frac{(1-p^{-2})^{2}}{{(1 - p^{-3})^{4/3}}} \cdot \frac{1}{\Reg_L^{1/6}}.
\]
Here
\begin{align*}
    \Reg_L &= 4 \log^2|r^4 - r^3 + 2r^2 - 2r + 1| \\
    & = 0.371341380445831860447262322286513810895\ldots,
\end{align*}
where \( r \) is the root of \(  X^6 - X^5 + 2X^4 - 3X^3 + 2X^2 - X + 1 \) near \( -0.148\ldots \pm -0.722i \).

As before, computing 
\[
    \log \prod_{p\in\CCC_3} \frac{(1-p^{-2})^{2}}{{(1 - p^{-3})^{4/3}}}
\]
by summing the \( p \leq 10^8 \) and estimating the error from the tail in \eqref{eqn:log1m1ontail}, gives
\begin{align*}
    \prod_{p\in\CCC_3} \frac{(1-p^{-2})^{2}}{{(1 - p^{-3})^{4/3}}} \in (0.766292551\ldots, 0.766292581\ldots) 
\end{align*}
on exponentiating. 
Hence \( C_{S} \in (0.41256118\ldots, 0.41256120\ldots) \), and so
\[
    C_{S} = 0.412561\ldots,
\]
correct to 6 digits.

\section{The Fourier coefficients of \texorpdfstring{$\Delta_{16}$}{Delta\_16} modulo 59}
\label{Haberland59}
The cusp forms for the full modular group are well-studied.
In the table below we give the full list of forms for which the corresponding weight-$k$ space of modular forms is $1$-dimensional, so that the Hecke eigenvalues are rational integers. Write $
\Delta_k$
for the normalized form of weight $k$. 
We have for example $$\Delta_{16} = q + 216q^2 - 3348q^3 + 13888 q^4 + 52110q^5 + O(q^6).$$
Following convention, we write $\Delta$ instead of 
$\Delta_{12}$. 

    \begin{table}[H]
\begin{center}
\begin{tabular}{|c||c|c|c|c|c|c|}
\hline
Weight $k$ & $12 $ & $ 16 $ & $ 18  $ & $ 20 $ & $ 22 $ & $ 26  $\\
Form & $\Delta$ & $\Delta E_4$ & $\Delta E_6  $ & $E_4^2\Delta $ & $E_4E_6 \Delta$ & $E_4^2E_6\Delta  $\\
\hline
\end{tabular}
\end{center}
\label{tab:formslist}
\end{table}
Many congruences for the 
Fourier coefficients
$a_n(\Delta_k)$
appeared in the literature.
This whole m\'enagerie of
results begged for a more theoretical explanation,
which only became available in the
early 1970s with 
Serre and Swinnerton-Dyer's analysis of modular forms modulo $\ell$ \cite{serrecong,S-D}.
They prove that congruences modulo a prime $\ell$ can only come in three flavors:
\begin{enumerate}[ref={(\roman*)}]
\item\label[type]{type1:eis} {\bf Eisenstein/abelian:} $a_n(\Delta_k)\equiv n^v \sigma_{k-1-2v}(n) \pmodi{\ell}$ for all $ (n,\ell)=1, $ and for some $ v\in\{0,1,2\}. $
\item\label[type]{type2:dih} {\bf Dihedral:} $a_n(\Delta_k)\equiv 0 \pmodi{\ell}$, whenever  
$\big(\frac{n}{\ell}\big) = -1$.
\item\label[type]{type3:except} {\bf Exceptional $\mathfrak S_4$:} $p^{1-k}a_p(\Delta_k)^2 \equiv\,0,1,2$ or $4 \pmodi{\ell}$ for all primes $ p\ne \ell. $
\end{enumerate} 
For the third
case, Serre and Swinnerton-Dyer showed only that \emph{if} a congruence of this type existed, 
then 
 necessarily $k=16$ and $\ell=59$; they were not able to establish it fully.
 This was eventually proved by 
 Haberland \cite{Haberland}, using Galois cohomological methods, in a series
of three papers.
Later, Boylan \cite{Boylan}, and Kiming and Verrill \cite{KiVe} gave different and
rather shorter proofs. Boylan also determined all such congruences for the coefficients of eta-product newforms.

The Euler-Kronecker constants for  
the sets \mbox{$S:=\{n:~\ell\nmid a_n(\Delta_k)\}$} for all congruences of \cref{type1:eis,type2:dih} were determined by Ciolan et al.~\cite{CLM}. For the one
congruence of \cref{type3:except}, this analysis was left in \cite{CLM} as a challenging open problem. Here we will 
resolve it. To this end, we set $a_n := a_{n}(\Delta_{16})$ and $S \coloneqq S_{\Delta_{16}, 59} = \{n \colon 59\nmid a_{n}\}$
for the rest of this section.

\subsection{The Galois representation and some relevant fields}\label{CCC}

A construction of Deligne attaches to $\Delta_{16}$ the Galois representation 
$$\rhobar = \rhobar_{\Delta_{16}, 59}: \Gal(\overline{\QQ}/\QQ) \longrightarrow \GL_2(\FF_{59}),$$
unramified outside $59$ and satisfying $\tr\rhobar(\Frob_p) = a_p$ 
as well as 
$\det \rhobar = \omega_{59}^{15}$  for~\mbox{$p \neq 59$}, where $$\omega_{59}: \Gal(\overline{\QQ}/\QQ) \twoheadrightarrow \Gal(\QQ(\zeta_{59})/\QQ \simeq \FF_{59}^\times$$ is the mod-$59$ cyclotomic character. 

We also consider the projectivization of $\rhobar$~:
$$\mathbb P \rhobar: \Gal(\overline{\QQ}/\QQ) \to \PGL_2(\FF_{59}).$$

As conjectured by Swinnerton-Dyer \cite[Cor.\,(iii) to Thm.~4]{Antwerp} and proved by 
Haberland~\cite{Haberland} (see also the exposition in Boylan \cite{Boylan}),
the image of $\mathbb P\rhobar$ in $\PGL_2(\FF_{59})$ is isomorphic to ${\mathfrak S}_4$. Let $L/\QQ$ be the fixed field of the kernel of $\mathbb P\rhobar$, so that $\Gal(L/\QQ) \simeq {\mathfrak S}_4$. 
Swinnerton-Dyer  \cite[p.\,35]{S-D} and Boylan \cite[Table 2]{Boylan} observe that $L$
is the splitting field of \mbox{$X^4 - X^3 - 7X^2 + 11X + 3$} (but note that \cite{S-D} has a typo, transposing two of the coefficients). 
With \texttt{Pari/GP}~\cite{PARI2}, one can compute the smallest defining polynomial for this number field%
{
\catcode`\^=11%
\footnote{%
Using the command  \texttt{polredabs(nfsplitting(X^4 - X^3 - 7*X^2 + 11*X + 3))}
}%
\catcode`\^=7
} giving  $L = \QQ(\alpha)$, where $\alpha$ is a root of  
\begin{equation}\label{eqn:fLpol}
f_L(x) = \begin{aligned}[t] 
X^{24} &- 6X^{22} + X^{20} + 91X^{18} + 118X^{16} - 157X^{14} - 360X^{12} \\
& + 17X^{10} + 312X^8 + 253X^6 + 95X^4 + 17X^2 + 1.
\end{aligned}\end{equation}
The fixed field of $\mathfrak A_4$, the index-2 subgroup of $G:=\Gal(L/\QQ) \simeq \mathfrak S_4$, is $K:=\QQ(\sqrt{-59})$. Then~$G$ has five conjugacy classes: $G_1,$ $G_2^\NR$, $G_2^\QR$, $G_3$, and $G_4$, where the subscript indicates the order of the elements and the superscripts $\NR$ and $\QR$ refer to whether or not the elements of the conjugacy class fix $\QQ(\sqrt{-59})$ --- in other words, $G_2^\NR$ is the set of $2$-cycles and $G_2^\QR$ is the three nontrivial elements of the Klein-4 subgroup of~$G$. Correspondingly, we define $5$ sets of primes: $\CCC_1$, $\CCC_2^\NR$, $\CCC_2^\QR$, $\CCC_3$, and $\CCC_4$ \label{c59def}:

{
\renewcommand{\arraystretch}{1.4}
\begin{center}
\begin{tabular}{l|c|c|l}
Set & Order of $\Frob_p$ & $p$ in $K$ & Sample primes\\ 
\hline\hline
$\CCC_1$ & 1 & splits & 163, 383, 439, 461, 1237, 1511, \ldots\\
\hline
$\CCC_2^\QR$ & 2 & splits &  17, 71, 139,  197, 223, 317, 373,\ldots\\
\hline
$\CCC_2^\NR$ & 2 & inert & 13, 23, 31,  37, 43, 47, 83,\ldots\\
\hline
$\CCC_3$ & 3  & splits &  3, 5, 7, 19, 29, 41, 53, \ldots\\
\hline
$\CCC_4$ & 4  & inert & 2, 11, 61, 67, 73, 89, 97, \ldots
\end{tabular}
\end{center}
}

We also write $\CCC_2$ for $\CCC_2^\QR \cup \CCC_2^\NR$, the primes $p$ with residue field $\FF_{p^2}$ in $L$.

Since $S$ is $L$-frobenian, the shape of $F_p(s)$ as a power series in  $p^{-s}$ should depend only on the conjugacy class of  
$\Frob_p$ in the extension $L/\QQ$.

\subsection{The Dirichlet series}
Since $S=S_{\Delta_{16},59}$ is a multiplicative set, our Dirichlet series has an Euler product as in  \eqref{eulerpro}.
Our goal in this section is to use the mod-$59$ Galois representation attached to $\Delta_{16}$ as well as the recurrence satisfied by the Hecke operators to determine  $i_S(p^m)$ --- that is, whether $p^m \in S$ for various $p$, $m$. Below $p \neq 59$ is a prime.

First, note that $a_{p^m}$ and hence $i_S(p^m)$ depends only on $m$ and the conjugacy class of $\Frob_p$ in $L$. Indeed for $m =0$ we have $a_1 =1$; for $m = 1$ we know that $a_p= \tr \rho (\Frob_p)$; and for~$m \geq 2$ we have the recurrence $a_{p^{m}} = a_p a_{p^{m-1}} - p^{k-1} a_{p^{m-2}}$ with characteristic polynomial $X^2 - a_p X + p^{k-1}$, so that the claim is true by induction.
Observe that the sequence $\{\tr \rho(\Frob_p)^m\}_m$ satisfies the same recurrence as $\{a_{p^m}\}_m$, and has the same~$m = 1$ term. But the $m = 0$ terms differ, as $\tr \rho(1) = 2$ but $a_1 = 1$. Proving the following lemma is then straightforward. 

\begin{Lem} Let $\alpha$, $\beta$ be the eigenvalues of $\rho(\Frob_p)$ with multiplicity\footnote{Although $\Frob_p$ and $\rho(\Frob_p)$ are only defined up to conjugacy, its  eigenvalues are well defined!}. Then
$$\tr \rho(\Frob_p^m) = \alpha^m + \beta^m$$ 
and
$$a_{p^m} = \frac{\alpha^{m+1}-\beta^{m+1}}{\alpha-\beta} = \alpha^m + \alpha^{m-1} \beta + \cdots + \alpha \beta^{m-1} + \beta^m.$$
\end{Lem}

Now let $b_m:= \displaystyle a_{p^m}^2/p^{(k-1)m}.$ Here $k = 16$. Recall that $i_S(p^m) = 1$ if and only if $59 \nmid b_m$.

\begin{Lem}
The sequence $\{b_m\}_m$ satisfies an order-3 linear recurrence with characteristic polynomial
$$ X^3 - (b_1 - 1) X^2 + (b_1 - 1) X -1.$$
\end{Lem}

\begin{proof} 
Since the sequence $\{a_{p^m}\}_m$ satisfies an order-2 linear recurrence with roots $\alpha$ and~$\beta$, the sequence $\{a_{p^m}^2\}_m$ satisfies an order-$3$ linear recurrence with roots $\alpha^2$, $\beta^2$, and $\alpha \beta$, and the sequence $b_m$ satisfies the linear recurrence with  roots $\alpha/\beta$, $\beta/\alpha$, and 1. Since $\alpha/\beta + \beta/\alpha = b_1 - 1$, the relevant characteristic polynomial of the recurrence is
\[
(X^2 - (b_1 - 1) X + 1)(X-1) = X^3 - (b_1 - 1) X^2 + (b_1 - 1) X -1.\qedhere
\]
\end{proof}

We now use the recurrence to determine $i_S(p^m)$ and hence $F_p(s)$ for primes in each conjugacy class as described in \cref{CCC}. 
Boylan \cite{Boylan} confirms Serre and Swinnerton-Dyer's computations of $b_1$-values, which appears in the second row below. We use the recurrence to fill in the rest. Note that $b_0 = 1$ for all primes $p$; moreover, $b_2 = (b_1-1)^2$. More generally, $b_m$ is in $\ZZ_{(59)}$, and below $\bar b$ denotes reduction modulo $\ell = 59$.

	 \begin{adjustbox}{rotate=0, center}
	\renewcommand{\arraystretch}{2}%
		\newcommand{\pr}[1]{%
			\ifnum#1=1%
			x%
			\else%
			x^{#1}%
			\fi%
		}%
		\begin{tabular}{c||c|c|c|c}%
		Prime set & $\CCC_1$ & $\CCC_2$ & $\CCC_3$ & $\CCC_4$ \\
		\hline
		$\bar b_1$ & 4 & 0 & 1 & 2 \\
		\hline
		\hline
		\makecell{\!\!Recurrence\! \\ \!\!polynomial\!} & $X^3 \- 3 X^2 \+ 3X \- 1$
		& $X^3 \+ X^2 \- X \- 1$
		& $X^3 \- 1$
		& $X^3 \- X^2 \+ X \- 1$\\
		\hline
		$\bar b_2$ & 9 & 1 & 0 & 1 \\
		\hline
		$\{\bar b_m\}_{m \geq 0}$ & 
		$\!1, 4, 9,\ldots,(m\+1)^2,\ldots\!$& 
		$\!1, 0, 1, 0, 1,\ldots\!$&
		$\!1, 1, 0, 1, 1,\ldots\!$ & 
		$\!1, 2, 1, 0, 1, 2, 1,\ldots\!$\\
		\hline
		 \makecell{$\bar b_m \neq 0$ \\ criterion}& 
		$m \not\equiv 58 \pmodi{59}$& 
		$m \not\equiv 1 \pmodi{2}$& 
		$m \not\equiv 2 \pmodi{3}$& 
		$m \not\equiv 3 \pmodi{4}$\\
		\hline
		 \makecell{$F_p(s)$, \\ $x \coloneqq p^{-s} $} &
		$\displaystyle \frac{1 \- \pr{58}}{(1 \- \pr{1})(1 \- \pr{59})}$&
		$\displaystyle \frac{1}{1 \- \pr{2}}$&
		$\displaystyle \frac{1 \- \pr{2}}{(1 \- \pr{3})(1 \- \pr{1})}$&
		$\displaystyle \frac{1 \- \pr{3}}{(1 \- \pr{4})(1 \- \pr{1})}$ 
	\end{tabular}
\end{adjustbox}

Finally, $\{a_{59^m}\}_m$ is a totally multiplicative sequence of numbers divisible by 59 as soon as $m \geq 1$, so that $F_{59}(s) = 1.$ Therefore,
{\[
F(s) = 
\begin{aligned}[t]
                \prod_{p \in \CCC_1} & \frac{1 - p^{-58s}}{(1 - p^{-s})(1 - p^{-59s})}
                \, \prod_{p \in \CCC_2}  \frac{1}{1 - p^{-2s}} \\
                &
                \, {} \times \prod_{p \in \CCC_3} \frac{1 - p^{-2s}}{(1 - p^{-3s})(1 - p^{-s})}
                \, \prod_{p \in \CCC_4} \frac{1- p^{-3s}}{(1 - p^{-s})(1 - p^{-4s})}.
    \end{aligned}
\]}

\subsection{Character table and the Artin \texorpdfstring{$L$}{L}-functions}
\label{sec:59:char} In this subsection we determine the Artin $L$-functions associated to the field $L$.

Recall the complete list of irreducible complex-valued representations of $\SS_4$: namely, the trivial representation $\one$; the sign representation $\sgn$; the $2$-dimensional representation inflated from the standard representation of $\SS_3$; the standard representation; and standard tensor $\sgn$.\footnote{Recall that the \emph{standard representation} of $\SS_n$ is the irreducible $(n-1)$-dimensional complement of the trivial representation in the $n$-dimensional representation permuting $n$ standard basis vectors.} It not difficult to determine the character table; see, e.g., Collins \cite{Collins}.
\renewcommand{\arraystretch}{1.4}
\begin{center}
\begin{tabular}{c|rrrrr}
character table & $G_1$ & $G_2^\QR$& $G_2^\NR$ & $G_3$ & $G_4$ \\ \hline
  $\one$ & 1 & 1 & 1 & 1 & 1 \\
  $\sgn$ & 1 & $1$ & $-1$ & 1 & $-1$ \\
  $\twod$ & 2 & 2 & 0 & $-1$ & 0 \\
   $\std$ 
   & 3 & $-1$ & $1$ & $0$ & $-1$ \\
  $\std \otimes \sgn$ 
  & 3 & $-1$ & $-1$ & 0 & $1$ \\
\end{tabular}
\end{center}

We now determine the $L$-function local factor for primes whose Frobenius elements land in each of these conjugacy classes. The table below gives, for each prime $p$ along the columns and representation~$(\rho, V)$ along the rows, the \mbox{$f_{\rho, p}(x) := \charpoly(\Frob_p \mid V^{I_p})(x)$}; the local factor of the $L$-function is then $L_p(\rho, s) = f(p^{-s})^{-1}$. For example, for $p \in \CCC^\NR_2$ (see~p.~\pageref{c59def}) and $\rho$ the $2$-dimensional irreducible representation of $\mathfrak S_4$ (which factors through its~$\mathfrak S_3$ quotient), we have $f_{p, \rho}(x) = 1 + x^2$, so that $L_p(\rho, s) = (1 + p^{-2s})^{-1}$.

	 \begin{adjustbox}{rotate=0, center}
\renewcommand{\arraystretch}{2}
\renewcommand{\frac}[2]{#2}
\newcommand{\pr}[1]{%
  \ifnum#1=1%
    x%
  \else%
    x^{#1}%
  \fi%
}%
    \begin{tabular}{c||c|c|c|c|c|c}
    $f_{p, \rho}(x)$
    & $p \in \CCC_1$ & $p \in \CCC_2^\QR$& $p \in \CCC_2^\NR$ & $p \in \CCC_3$ & $p \in \CCC_4$ & $p = 59$ 
\\ \hline\hline
  $\one$ & 
    $\displaystyle \frac{1}{1-{\pr{1}}}$ & 
    $\displaystyle \frac{1}{1-{\pr{1}}}$ & 
    $\displaystyle \frac{1}{1-{\pr{1}}}$ &
    $\displaystyle \frac{1}{1-{\pr{1}}}$ &
    $\displaystyle \frac{1}{1-{\pr{1}}}$ &
    $\displaystyle \frac{1}{1-\pr{1}}$
  \\
 $\sgn$ &
    $\displaystyle \frac{1}{1-{\pr{1}}}$ &
    $\displaystyle \frac{1}{1-{\pr{1}}}$ &
    $\displaystyle \frac{1}{1+{\pr{1}}}$ &
    $\displaystyle \frac{1}{1-{\pr{1}}}$ &
    $\displaystyle \frac{1}{1+{\pr{1}}}$ &
    $1$ \\
  $\twod$ &
  $\displaystyle \frac{1}{(1-{\pr{1}})^2}$ &
  $\displaystyle \frac{1}{(1-{\pr{1}})^2}$ &
  $\displaystyle \frac{1}{1+{\pr{2}}}$ & 
  $\displaystyle \frac{1}{1+{\pr{1}}+{\pr{2}}}$ &
  $\displaystyle \frac{1}{1-{\pr{2}}}$ &
  $\displaystyle \frac{1}{1-{\pr{1}}}$
 \\
 $\std$ &
 $\displaystyle \frac{1}{(1-{\pr{1}})^3}$ &
 $\displaystyle \frac{1}{1+{\pr{1}}-{\pr{2}}-{\pr{3}}}$ &
 $\displaystyle \frac{1}{1-{\pr{1}}-{\pr{2}}-{\pr{3}}}$ &
 $\displaystyle \frac{1}{1-{\pr{3}}}$ &
 $\displaystyle \frac{1}{1+{\pr{1}}+{\pr{2}}+{\pr{3}}}$ &
 $1$
\\
 $ \std \otimes \sgn $ & 
 $\displaystyle \frac{1}{(1-{\pr{1}})^3}$ &
  $\displaystyle \frac{1}{1+{\pr{1}}-{\pr{2}}-{\pr{3}}}$ &
  $\displaystyle \frac{1}{1+{\pr{1}}-{\pr{2}}+{\pr{3}}}$ & 
  $\displaystyle \frac{1}{1-{\pr{3}}}$ &
  $\displaystyle \frac{1}{1-{\pr{1}}+{\pr{2}}-{\pr{3}}}$ &
  $\displaystyle \frac{1}{1+{\pr{1}}}$
    \\
    \end{tabular}
\end{adjustbox}

We compute each characteristic polynomial $f_{\rho, p}(x)$ from the character table using the relationship between coefficients $e_i(M)$ of $\charpoly(M) = 1 - e_1(M) x + e_2(M) x^2 + \cdots$ of a matrix $M$ and traces of powers of $M$ that one can derive from Newton's identities.
For example, with $M = \rho(g)$ for $g$ any choice of $\Frob_p$ with $p \in \CCC_4$ and $\rho = \std$, we have~$g, g^3 \in G_4$ and $g^2 \in G_2^\QR$, so that $\tr M^i = -1$ for $i = 1,2,3$. It follows that 
\begin{align*}e_1(M) &= \tr M = -1, \\
e_2(M) &= \frac{1}{2}(\tr M)^2 - \tr M^2) = \frac{1}{2}\big((-1)^2 - (-1)\big) = 1, \text{ and }\\ 
e_3(M) &= \frac{1}{6}\big((\tr M)^3 \!- 3(\tr M^2) (\tr M) + 2 \tr M^3\big) = \frac{1}{6}\big((-1)^3 \! - 3(-1) (-1) + 2 (-1)\big) = -1.
\end{align*}
In each case, the first equality is a general formula.

The table also includes the local factor at $\ell = 59$. For each representation $(\rho, V)$, we compute the characteristic polynomial of the action of $\rho(\Frob_p)$ on $V^{I_{59}}$, the subspace of~$V$ fixed by the at-59 inertia (which is a priori only defined up to conjugation). Since the ideal $(59)$ factors in $\OO_L$ into $\pp_1^4 \pp_2^4 \pp_3^4$, where each $\pp_i$ has residue degree $2$ (computed in \texttt{Sage}  \cite{SAGE} or \texttt{Pari/GP} \cite{PARI2}), we see that the decomposition groups $D_{\pp_i}$ are the three conjugate copies of $D_4$ sitting as the $2$-Sylow subgroups inside $\SS_4$. To fix ideas, we may assume that $D_{59} = \{e, (1234), (13)(24), (1432), (13), (24), (14)(23), (12)(34)\}$. Since the inertia is tame, it is cyclic, so that $I_{59} = \{e, (1234), (13)(24), (1432)\}$. The complement of $I_{59}$ in $D_{59}$ is the coset of Frobenius elements. Now we  compute the local factor for each~$\rho.$ 
\begin{itemize}[itemsep = 3pt, topsep = -3pt, left=5em] 
\item[\fbox{$\rho = \one$}] We have $V^{I_{59}}$ one-dimensional. Frobenius acts trivially. 
\item[\fbox{$\rho = \sgn$}] We have $V^{I_{59}} = \{0\}$, as $G_4$ acts nontrivially. 
\item[\fbox{$\rho = \twod$}] We have $\dim V^{I_{59}} = 1.$ Indeed, $\rho$ is faithful on the $\SS_3$ quotient of $\SS_4$ with the $(2, 2)$-cycles acting trivially, so that the $G_4$ elements in $I_{59}$ act by a flip about a line fixed by $I_{59}$. Frobenius acts trivially.
\item[\fbox{$\rho = \std$}] We have $V^{I_{59}} = \{0\}$: indeed, the characteristic polynomial \mbox{$x^3 + x^2 + x + 1$} of the~$G_4$ elements does not have $1$ as a root.  
\item[\fbox{$\rho = \std \otimes \sgn$}] We have $\dim V^{I_{59}} = 1$: if $b_1, b_2, b_3, b_4$ is the index basis for the permutation representation of $\SS_4$, and $\{b_1 - b_2, b_2 - b_3, b_3 - b_4\}$ is the basis for the standard subrepresentation, then $V^{I_{59}}$ (as fixed above) is spanned by $b_1 - b_2 + b_3 - b_4$, and Frobenius acts by $-1$. 
\end{itemize}

\subsection{Expressing \texorpdfstring{$F$}{F} in terms of Artin \texorpdfstring{$L$}{L}-functions}\label{sec:f3toartin}

In contrast to the previous case, where we were able to identify the Artin $L$-function contribution directly, now we need to follow a more systematic approach to express \( F(s) \) via the 5 distinct Artin $L$-functions above (up to a factor absolutely convergent around \( s = 1 \)).

Each of the sets of primes \( \CCC_1, \CCC_2^{\mathrm{QR}}, \CCC_2^{\mathrm{NR}}, \CCC_3, \CCC_4 \) has positive natural density, so to guarantee convergence of an Euler product of the form
\[
    \prod_{p \in \CCC^\ast_j} \Big( 1 + \frac{b_1}{p^s} + \frac{b_2}{p^{2s}} + O({p^{-3s}}) \Big) \,,
\]
at \( s = 1 \), it suffices to require \( b_1 = 0 \).  Note also that
\begin{equation}\label{eqn:plinear}
    \Big( 1 + \frac{b_1}{p^s} + \frac{b_2}{p^{2s}} + O({p^{-3s}}) \Big) \cdot \Big( 1 + \frac{b'_1}{p^s} + \frac{b'_2}{p^{2s}} + O(p^{-3s}) \Big) = 1 + \frac{b_1 + b'_1}{p^s} + O(p^{-2s}) \,,
\end{equation}
so that the condition on the \( p^{-s} \)-coefficient is linear, when multiplying the local factor at $p$ of \( L \)-functions with an Euler product.  Our goal is therefore to find exponents \( k_0,k_1,\ldots,k_5 \) associated to \( F \) and each irreducible representation \( \rho_1 = \one,\) $\rho_2 = \sgn$, $\rho_3 = \twod$, $\rho_4 = \std,$ $\rho_5 = \std \otimes \sgn $, so that in
\[
    F(s)^{k_0} \prod_{i=1}^5 L(\rho_i, s)^{k_i} = \prod_j \prod_{p \in \CCC_j^\ast} \Big( 1 + b_{j1}\,p^{-s} + O(p^{-2s}) \Big) \,,
\]
the coefficient $b_{j1}$ of \( \frac{1}{p^s} \) vanishes in every factor, i.e. $b_{j1}=0$.  The right-hand side will then define the absolutely convergent (in $s=1$) factor $H(s)$ that we seek.

Considering first the set \( \CCC_1 \), we see that in \( F(s) \), this factor is
\[
    F(s) \colon \quad \prod_{p\in \CCC_1} \frac{1 - p^{-58s}}{(1 - p^{-s})(1 - p^{-59s})} = \prod_{p\in C_1} \Big( 1 + p^{-s} + O(p^{-2s})  \Big) \,.
\]
Likewise, in each of \( L_\Q(\rho_i, s) \), we read off the factor of \( \prod_{p\in \CCC_1} \) from the first column of the table in \cref{sec:59:char}.
\begin{align*}
    L(\rho_1, s) = L(\one, s) & \colon \quad \prod_{p\in \CCC_1} (1-p^{-s})^{-1} = \prod_{p\in \CCC_1} \big( 1 + p^{-s} + O(p^{-2s}) \big) \,, \\
    L(\rho_2, s) = L(\sgn, s) & \colon \quad \prod_{p\in \CCC_1} (1-p^{-s})^{-1} = \prod_{p\in \CCC_1} \big( 1 + p^{-s} + O(p^{-2s}) \big) \,, \\
    L(\rho_3, s) = L(\twod, s) & \colon \quad \prod_{p\in \CCC_1} (1-p^{-s})^{-2} = \prod_{p\in \CCC_1} \big( 1 +2p^{-s} + O(p^{-2s}) \big) \,, \\
    L(\rho_4, s) = L(\std, s) & \colon \quad \prod_{p\in \CCC_1} (1-p^{-s})^{-3} = \prod_{p\in \CCC_1} \big( 1 + 3p^{-s} + O(p^{-2s}) \big) \,, \\
    L(\rho_5, s) = L(\std \otimes \sgn, s) & \colon \quad \prod_{p\in \CCC_1} (1-p^{-s})^{-3} = \prod_{p\in \CCC_1} \big( 1 + 3p^{-s} + O(p^{-2s}) \big) \,.
\end{align*} 
Note that the coefficient of \( p^{-s} \) is exactly the value of the character \( \chi_{\rho_i} \) on the conjugacy class \( \CCC_1 \).  This is not a coincidence: the local factor of $L(\rho, s)$ at a good prime $p$ is $$\frac{1}{\charpoly(\Frob_p|\rho)(p^{-s})} = \frac{1}{1 - \tr(\Frob_p| \rho)/p^{s} + O(p^{-2s}) } = 1 + \frac{\chi_\rho(\Frob_p)}{p^{s}} + O(p^{-2s}).$$

For the factor \( \prod_{p \in \CCC_1} \), we need to impose the condition
\[
    k_0 \cdot 1 + k_1 \cdot 1 + k_2 \cdot 1 + k_3 \cdot 2 + k_4 \cdot 3 + k_5 \cdot 5 = 0 \,,
\]
to force it to converge when \( s = 1 \).

Continuing for the subsequent prime classes, we obtain the following system of equations
\[
\left(\!\!
\begin{array}{c|ccccc}
 1 & 1 & 1 & 2 & 3 & 3 \\
 0 & 1 & 1 & 2 & -1 & -1 \\
 0 & 1 & -1 & 0 & 1 & -1 \\
 1 & 1 & 1 & -1 & 0 & 0 \\
 1 & 1 & -1 & 0 & -1 & 1 \\
\end{array}\!\!
\right) \cdot
\left(\!\!
\begin{array}{c}
k_0 \\ 
k_1 \\
 k_2 \\
 k_3 \\
 k_4 \\
 k_5
\end{array}\!\!
\right) = \uline{\mathbf{0}} \,.
\]
This system has a 1-dimensional solution, spanned by
\[
    \big(\,\,\begin{matrix} 8 & -5 & -1 & 2 & 1 & -3 \end{matrix}\,\,\big)^\top \,.
\]
This tells us that
\[
H(s) \coloneqq 
F(s)^8 
L_\Q(\one, s)^{-5}
L_\Q(\sgn, s)^{-1} 
L_\Q(\twod, s)^{2} 
L_\Q(\std, s)^1
L_\Q(\std \otimes \sgn, s)^{-3} 
\]
is convergent as \( s \to 1 \).  

Since \( L_\Q(\one, s) = \zeta(s) \), we obtain the following
\begin{equation}\label{eqn:ls3:artin}
    F(s)^8 = \begin{aligned}[t]
    & \frac{\zeta(s)^5 L_\Q(\sgn, s) L_\Q(\std \otimes \sgn, s)^3}{L_\Q(\twod , s)^2 L_\Q(\std, s)} \cdot H(s)
    \end{aligned},
\end{equation}
where
\begin{equation}
    \label{G3s}
    H(s) = \begin{aligned}[t]
    & (1 - 59^{-2s})^3 
    \prod_{p \in \mathcal{C}_1} \frac{(1 - p^{-58s})^8}{(1 - p^{-59s})^8}
                \cdot \prod_{p \in \mathcal{C}_2}  \frac{1}{(1 - p^{-2s})^4} \\
    &
                \cdot \prod_{p \in \mathcal{C}_3} \frac{(1 - p^{-2s})^8}{(1 - p^{-3s})^8}
                \cdot \prod_{p \in \mathcal{C}_4}  \frac{(1 - p^{-3s})^8}{(1 - p^{-2s})^4 (1 - p^{-4s})^6}
    \end{aligned}
\end{equation}
is absolutely convergent for \( \Re(s) > 1/2 \).  This 
shows that \( \delta = 5/8  \) in this case.

\subsection{Reformulation in terms of Dedekind \texorpdfstring{$\zeta$}{zeta}-functions}

Following \cite{Antwerp}, we recall the tower of fixed fields 
\(
    \Q \subset K = \Q(\sqrt{-59}) \subset M \subset L 
\),
corresponding to the subgroups
\(
    \mathfrak{S}_4 \supset \mathfrak{A}_4 \supset V \supset \{ e \} 
\).  Here \( L \) is the splitting field of \( X^4 + X^3 - 7X^2 - 11X + 3 \), as above, and \( M \) is the splitting field of \( X^3 + 2X - 1 \).  
Denote by \( E \) the field with defining polynomial \( X^4 + X^3 - 7X^2 - 11X + 3 \); one sees that \( E = K^{\mathfrak{S}_3} \) is the fixed field corresponding to a copy of  \( \mathfrak{S}_3 \) inside $\mathfrak S_4$.  
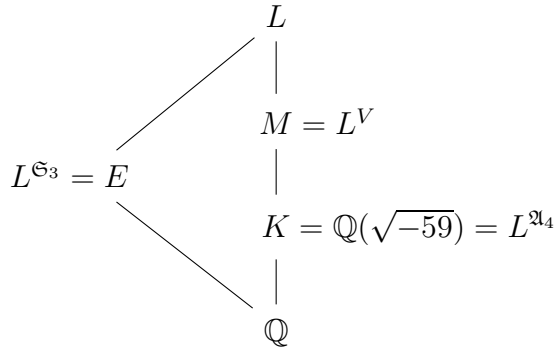
\begin{center}
    \begin{tikzpicture}[scale=0.7]

    \node (L) at (0,6) {$L$};
    \node (M) at (0,4) {$M\mathrlap{{} = L^V}$};
    \node (K) at (0,2) {$K\mathrlap{{} = \Q(\sqrt{-59}) = L^{\mathfrak A_4}}$};
    \node (Q) at (0,0) {$\Q$};
    \node (E) at (-3,3) {$\mathllap{L^{\mathfrak{S}_3} = {}} E$};
    
    \draw($(L.south) + (0,-2pt)$)--($(M.north) + (0,2pt)$);
    \draw($(M.south) + (0,-2pt)$)--($(K.north) + (0,2pt)$);
    \draw($(K.south) + (0,-2pt)$)--($(Q.north) + (0,2pt)$);
    \draw($(L.south west) + (0,0pt)$)--($(E.north) + (0,0pt)$);
    \draw($(E.south) + (0,-2pt)$)--($(Q.north west) + (0,0pt)$);

    \end{tikzpicture}
    \end{center}

The following table summarises this, and gives the defining polynomials of these fields.  We can use the Dedekind $\zeta$-functions of these fields to reëxpress the Artin $L$-functions we found above. 
\begin{center}
\begin{tabular}{c|c|c|c}
    Degree & Field name & Subgroup & Defining polynomial \\ \hline
    1 & \( \QQ \) & ${\mathfrak S}_4$ & $X$ \\
    2 & \( K = \QQ(\sqrt{-59}) \) & $\mathfrak A_4$ & $X^2 + 59$ \\ 4 & \( E = K^{{\mathfrak S}_3} \) & ${\mathfrak S}_3$ & $X^4 + X^3 - 7X^2 - 11X + 3$ \\
    6 & \( M = L^V \) & $V$ & $
    X^6 + 12X^4 + 36X^2 + 59$ \\
    24 & $L$ & $\{e\}$ & $
    f_L(X)$ (as in \eqref{eqn:fLpol} above)
\end{tabular}
\end{center}

It is straight-forward (using Frobenius reciprocity, say) to check that
\begin{align*}
    \Ind_{\mathfrak{S}_4}^{\mathfrak{S}_4} \one &= \one \\
    \Ind_{A_4}^{{\mathfrak{S}_4}} \one &= \one + \sgn \\
    \Ind_{\mathfrak{S}_3}^{\mathfrak{S}_4} \one  &= \one + \std \\
    \Ind_{V}^{\mathfrak{S}_4}  \one &= \one + \sgn + 2 \cdot \twod \\
    \Ind_{\{e\}}^{\mathfrak{S}_4}  \one &= \mathrm{regular} = \one + \sgn + 2 \cdot \twod + 3 \cdot \std + 3 \cdot (\std \otimes \mathrm{sgn})
\end{align*}
Hence
\begin{align*}
    \one &= \Ind_{\mathfrak{S}_4}^{\mathfrak{S}_4} \one \\
    \sgn &= \Ind_{A_4}^{\mathfrak{S}_4} \one - \Ind_{\mathfrak{S}_4}^{\mathfrak{S}_4} \one \\
    \mathrm{std} &= \Ind_{\mathfrak{S}_3}^{\mathfrak{S}_4} \one - \Ind_{\mathfrak{S}_4}^{\mathfrak{S}_4} \one \\
    \twod &= \frac{1}{2} \Ind_{V}^{\mathfrak{S}_4} \one - \frac{1}{2} \Ind_{A_4}^{\mathfrak{S}_4} \one \\
    \std \otimes \sgn &= \frac{1}{3} \Ind_{\{e\}}^{\mathfrak{S}_4} \one - \frac{1}{3} \Ind_{V}^{\mathfrak{S}_4} - \Ind_{\mathfrak{S}_3}^{\mathfrak{S}_4} \one + \Ind_{\mathfrak{S}_4}^{\mathfrak{S}_4} \one
\end{align*}
Using the properties of Artin $L$-functions from \cref{sec:Artin},  
we immediately have the following expressions of the Artin $L$-functions via Dedekind $\zeta$-functions of intermediate fields of \( L/\Q \): 
\begin{align*}
     L_\QQ(\one, s) &= \zeta(s) \\
     L_\QQ(\sgn, s) &= \zeta_{K}(s)/\zeta(s) \\
     L_\QQ(\twod, s)^2 &= \zeta_M(s)/\zeta_K(s) \\
     L_\QQ(\std, s) &= \zeta_E(s) / \zeta(s) \\
     L_\QQ(\std \otimes \sgn, s)^3 &= \frac{\zeta(s)^3 \zeta_L(s)}{\zeta_E(s)^3 \zeta_M(s)} \,.
\end{align*}

Substituting these identities into \eqref{eqn:ls3:artin}, we alternatively obtain
\begin{equation}\label{eqn:ls3:dedekind}
    F_{S}(s)^8 = \frac{\zeta_{K}(s)^2 \zeta_{L}(s) \zeta(s)^8}{\zeta_{E}(s)^4 \zeta_{M}(s)^2} \cdot H(s),
\end{equation}
with $H(s)$ as in \eqref{G3s}.

	\begin{rk}
		The fields in the above table are not those which \cref{dedekindzetamain} envisages.  \cref{dedekindzetamain} implies that we can take number fields \( L_i \subset K \) with \( \Gal(K/L_i) \) cyclic.  The cyclic subgroups (up to conjugacy, giving isomorphic fixed fields) in \( \mathfrak{S}_4 \) are generated by the identity \( e \), a two-cycle \( (12) \), a double transposition \( (12)(34) \), a three-cycle \( (123) \) or a 4-cycle \( (1234) \), giving subgroups of order \( 1, 2, 2, 3 \) and \( 4 \) respectively.  This gives fields of degree \( 24, 12, 12, 8 \) and \( 6 \) respectively; the Dedekind $\zeta$-functions of these fields can be used to express the Artin $L$-functions in \eqref{eqn:ls3:artin}.  This follows in particular from following inductions (checked using Frobenius reciprocity)
		\begin{align*}
		    \one &= \phantom{{}+} \tfrac{1}{2} \Ind_{(1234)}^{\mathfrak{S}_4} \one +  \tfrac{1}{2}  \Ind_{(123)}^{\mathfrak{S}_4} \one +  \tfrac{1}{2}  \Ind_{(12)}^{\mathfrak{S}_4} \one -  \tfrac{1}{2} \Ind_{e}^{\mathfrak{S}_4} \one \\
		    \sgn &= -\tfrac{1}{2} \Ind_{(1234)}^{\mathfrak{S}_4} \one +  \tfrac{1}{2}  \Ind_{(123)}^{\mathfrak{S}_4} \one -  \tfrac{1}{2}  \Ind_{(12)}^{\mathfrak{S}_4} \one +  \tfrac{1}{2} \Ind_{(12)(34)}^{\mathfrak{S}_4} \one \\
		    \twod & = -\tfrac{1}{2} \Ind_{(123)}^{\mathfrak{S}_4} \one +  \tfrac{1}{2}  \Ind_{(12)(34)}^{\mathfrak{S}_4} \one \\
		    \std & = -\tfrac{1}{2} \Ind_{(1234)}^{\mathfrak{S}_4} \one +  \tfrac{1}{2}  \Ind_{(12)}^{\mathfrak{S}_4} \one \\
		    \std \otimes \sgn & =  \phantom{{}+}  \tfrac{1}{2} \Ind_{(1234)}^{\mathfrak{S}_4} \one -  \tfrac{1}{2}  \Ind_{(12)}^{\mathfrak{S}_4} \one  - \tfrac{1}{2}  \Ind_{(12)(34)}^{\mathfrak{S}_4} \one +   \tfrac{1}{2}  \Ind_{e}^{\mathfrak{S}_4} \one
		\end{align*}
		However, invoking several additional fields of degree $\geq8$ increases the computational complexity superfluously, and moves us away from already (implicitly) defined fields like \( E = K^{\mathfrak{S}_3} \).  We hence prefer the above list.
	\end{rk}

\subsection{Computation of the Euler-Kronecker constant}

Logarithmic differentiation of \eqref{eqn:ls3:dedekind} gives
\[
8 \frac{F'}{F}(s) = 8 \frac{\zeta'}{\zeta}(s) + 2 \frac{\zeta_{K}'}{\zeta_{K}}(s) + \frac{\zeta_L'}{\zeta_L}(s) - 4 \frac{\zeta_E'}{\zeta_E}(s) - 2 \frac{\zeta_M'}{\zeta_M}(s) + \frac{H'}{H}(s) \,,
\]
leading to
\[
    8 \gamma_{S} = 8 \gamma + 2 \gamma_{K} + \gamma_L - 4 \gamma_E - 2 \gamma_M + \frac{H'}{H}(1) \,.
\]
The Ihara bounds \eqref{eqn:ihara} with \( x = 10^8 \) give
\begin{align*}
    \gamma_{K} &\in (0.04396032\ldots, 0.04422743\ldots), \\
    \gamma_{L} &\in (2.02971392\ldots, 2.03401336\ldots), \\
    \gamma_{E} &\in (0.83356048\ldots, 0.83453128\ldots), \\
    \gamma_{M} &\in (1.33846236\ldots, 1.33894626\ldots), \\
\end{align*}
For the quadratic field \( K = \Q(\sqrt{-59})\), 
a much more accurate and unconditional result is available \cite{ibid}, namely
$\gamma_{K}  = 0.044084767022407053675495516436691\ldots$

Logarithmic differentiation of \( H(s) \) leads to
\[
    \frac{H'}{H}(1) =
    \begin{aligned}[t]
    & \frac{6 \log 59}{59^2-1}
    + \sum_{p \in \mathcal{C}_1} \Big( \frac{8\cdot58 }{p^{58}-1} - \frac{8 \cdot 59}{p^{59}-1} \Big) \log p
    - \sum_{p \in \mathcal{C}_2^{QR} \cup \mathcal{C}_2^{NR}}  \frac{8 \log p }{p^2-1} \\
    & + \sum_{p \in \mathcal{C}_3} \Big( \frac{16 }{p^{2}-1} - \frac{24}{p^{3}-1} \Big) \log p
    + \sum_{p \in \mathcal{C}_4} \Big( - \frac{8 }{p^{2}-1} + \frac{24}{p^{3}-1} - \frac{24}{p^4 - 1} \Big) \log p.
    \end{aligned}
\]
Explicitly summing the contribution of the primes \( p \leq 10^8 \), and estimating the tail via \eqref{trivialtail} gives
\begin{equation}\label{eqn:G3poG3at1}
    \frac{H'}{H}(1) \in (1.64466270\ldots, 1.64466337\ldots) \,.
\end{equation}
Overall, we find $\gamma_{S} \in (0.29553172\ldots, 0.29667561\ldots)$ and so 
$$\gamma_{S} = 0.29\ldots $$ is correct up to 2 digits (even just using the Ihara bounds \eqref{eqn:ihara} for \( K = \Q(\sqrt{-59}) \)).

\begin{rk}[Increased precision from heuristic \texttt{Pari/GP} calculations of \( \gamma_L, \gamma_E \) and \( \gamma_{M} \)]
One can apply the approach from \cref{app:higherprec} to calcualte \(  \gamma_L, \gamma_E \) and \( \gamma_M \) to higher precision using \texttt{Pari/GP}'s numerical $L$-function routines.
One obtains
\begin{align*}
    \gamma_L  \overset{?}{=} 0.0314402432715970302501248\ldots \,, \\
    \gamma_E \overset{?}{=} 0.8339340133907066887364345\ldots \,, \\
    \gamma_M  \overset{?}{=} 0.3387127640443338396608939\ldots \,. \\
\end{align*}
From these and the estimate on \( \frac{H'}{H}(1) \) from \eqref{eqn:G3poG3at1}, we obtain
\[
    \gamma_{S} \overset{?}{\in} ( 0.29610452\ldots , 0.29610461\ldots ) \,,
\]
Hence
\[
    \gamma_{S} \overset{?}{=} 0.29(6104)\ldots
\]
with 4 extra \emph{uncertified} digits.
\end{rk}

\subsection{Calculation of \texorpdfstring{\( C_{S} \)}{C\_S}}   Application of \eqref{residu}, with \( \delta = 5/8 \) to \eqref{eqn:ls3:dedekind} gives
\[
     C_{S_3} = \frac{1}{\Gamma(1/3)} \frac{(\Res_{s=1} \zeta_{K}(s))^{1/4} (\Res_{s=1} \zeta_{L}(s))^{1/8}}{(\Res_{s=1} \zeta_{E}(s))^{1/2} (\Res_{s=1} \zeta_{M}(s))^{1/4}} \cdot H(1)^{1/8}\,.
\]
From the analytic class number formula, we have
\begin{align*}
    \Res_{s=1} \zeta_{K}(s) & = \frac{3\pi}{\sqrt{59}} \\
    \Res_{s=1} \zeta_{E}(s) &= \frac{4\pi}{\sqrt{59^3}} \Reg_E \,, \quad \Reg_E  = 20.5225859936222504507764283\ldots \\
    \Res_{s=1} \zeta_{M}(s) &= \frac{4 \pi^3}{\sqrt{59^3}} \Reg_M \,, \quad \Reg_M = 1.87697523076402072994449005\ldots \\
    \Res_{s=1} \zeta_{L}(s) &= \frac{2^{12} \pi^{12}}{59^9} \Reg_L \,, \quad \Reg_L =  1235592.8728873038384079887260792 \ldots \\
\end{align*}
The regulators were computed using the \texttt{bnfinit} routine from \texttt{Pari/GP}, and certified using \texttt{bnfcertify} to remove the assumption of GRH used in fast algorithms for computing fundamental units and class groups.  For details see \cref{app:reg}.

Hence
\[
    C_{S} = \frac{1}{\Gamma(5/8)} \frac{3^{1/4} \sqrt{\pi}}{(2\cdot59)^{1/8}} \frac{\Reg_L^{1/8}}{\Reg_E^{1/2} \Reg_M^{1/4}} \cdot H(1)^{1/8}.
\]
As before, estimation of \( \log H(1) \), using \eqref{eqn:log1m1ontail} with \( p \leq 10^8 \) gives
\begin{align*}
    \log H(1) &\in (-0.52586694\ldots, -0.52586662\ldots) \\
     H(1) &\in (0.59104274\ldots, 0.59104293\ldots) 
\end{align*}
Hence $C_{S} \in (0.91331343\ldots, 0.91331347\ldots)$, and so
$$C_{S} = 0.9133134\ldots,$$ correct to 7 digits.

\section{General theorem} 
\label{generalsection}

The following theorem is our main theoretical result.

We first set up some notation. Let $f = \sum_n a_n(f)q^n$ be a normalized eigenform of some weight and level, with Hecke eigenfield $K := \QQ(\{a_n(f): n\})$. Let $\lambda$ be a prime of $K$, and $v_\lambda$ the corresponding valuation. Let $$S := S_{f, \lam} := \{n \geq 1: v_\lam\bigl(a_n(f)\Bigr) = 0\}.$$ Then $S$ is a multiplicative set, so that the Dirichlet series 
$$F = F_{S}(s) = F_{f, \lam}(s) = \sum_{n \in S} \frac{1}{n^s}$$
has an Euler product expansion. 

\begin{Thm}\label{mainthm}
There is a number field $L$, Galois over $\QQ$, so that $S$ is \emph{$L$-frobenian}: for all but finitely many primes $p$, whether $p$ is in $S$ is determined only by the conjugacy class of $\Frob_p$ in $L/\QQ$. Moreover, $L$ satisfies the following additional properties. 
\begin{enumerate}[ref={\theThm(\roman*)},leftmargin=2em, itemsep = 3pt, topsep = -5pt]
\item \label[theorem]{artinlmain} Let $G = \Gal(L/\QQ)$, and let $\rho_1, \ldots, \rho_g$ be the list of irreducible complex representations of $G$. There is a corresponding sequence of rational numbers $q_1, \ldots, q_g$ so that the Dirichlet series $F(s)$ has the form 
$$F(s) = J(s) \prod_{i = 1}^g L (\rho_i, s)^{q_i}$$
for some nonvanishing holomorphic function $J(s)$ uniformly bounded
on \mbox{$\Re(s)>1-\varepsilon$} for some $\varepsilon>0$.
\item \label[theorem]{dedekindzetamain} There is a finite list of subfields $L_1, \ldots, L_d$ of $L$ and corresponding rational numbers $r_1, \ldots, r_d$, so that 
$$F(s) =  H(s)\prod_{i} \zeta_{L_i}(s)^{r_i}$$
for some nonvanishing holomorphic function $H(s)$ uniformly bounded
on \mbox{$\Re(s)>1-\varepsilon$} for some $\varepsilon>0$.
Moreover, the $L_i$ may be chosen to be  pairwise nonconjugate over~$\QQ$, with $\Gal(L/L_i)$ cyclic. 
\end{enumerate}
\end{Thm}
By logarithmic differentiation we immediately obtain the following 
\begin{Cor}
$$\gamma_S=\frac{J'}{J}(1)+\sum_{j=1}^g q_j\frac{L'(\rho_j,1)}{L(\rho_j,1)}=\frac{H'}{ H}(1)+\sum_{j=1}^g r_j\frac{\zeta_K'}{\zeta_K}(1),$$
the last expression allowing one to determine $\gamma_S$, under GRH, with a modest numerical precision using the method of
Ihara (\cref{sec:Dedekind}).
\end{Cor}

\begin{proof}[Proof of \cref{mainthm}]
First of all, since $S$ is a  multiplicative set, $F(s)$ has an Euler product \eqref{eulerpro}.
As is typical in these situations, it is the function $i_S(p)$ that determines the behavior of $F(s)$ as $s \to 1$. Indeed, we can write
$$F(s)=\prod_p\left(1+\frac{i_S(p)}{p^s}+\frac{i_S(p^2)}{p^{2s}}+O(p^{-3s})\right)$$
We now give a Galois-theoretic interpretation of the function $i_S(p)$, and use this to deduce the appropriate combination of Artin $L$-functions which will differ multiplicatively from \( F(s) \) by an absolutely convergent factor. Indeed, it is enough to ensure that, except for possibly finitiely many primes $p$, the local Euler factor of this combination equals $1+i_S(p)\,p^{-s}+O(p^{-2s})$.

Let $\FF$ be the field~$\OO_K/\lam$. 
By a theorem of Deligne, we can attach to $f$ modulo $\lam$ a continuous Galois representation 
$$\rho: \Gal(\overline\QQ/\QQ) \to \GL_2(\FF),$$
which is unramified outside some finite set $T$ of primes and so that for every $p \not \in T$ we have 
$\tr \rho(\Frob_p) = \bar a_p(f)$, where $\bar a_p(f)$ is the image of $a_p(f)$ in $\FF$. Of course $i_S(p) = 1$ if and only 
if $\bar a_p(f) \neq 0$.

Let~$L$ be the fixed field of $\ker \rho$, so that $\rho$ factors through $G := \Gal(L/\QQ).$\footnote{Alternatively, 
we can let $L$ be the smaller field fixed by the kernel of the projective representation $\mathbb P \rho: \Gal(\overline \QQ/\QQ) \to \PGL_2(\FF)$, and still set $G := \Gal(L/\QQ).$} By the Chebotarev density theorem, Frobenius conjugacy classes are equidistributed in $G$. In other words, there is class function\footnote{Recall that a \emph{class function} on finite group is $G$ is a map $G \to \CC$ that is constant on conjugacy classes.
} 
$\psi = \psi_S: G \to \{0, 1\} \subset \QQ$ defined by \mbox{$\psi(\Frob_p) = i_S(p)$} for all $p \not\in T$. 
(The existence of $\psi$ is another way of saying that $S$ is $L$-frobenian.)

Let $\Lambda := \{\lambda_1, \ldots, \lambda_g\}$ be an ordered list of the class functions that are characters of  irreducible (complex) representations of $G$: that is $\lambda_i = \tr \rho_i$.
We claim that $\psi$ is in the $\QQ$-span of $\Lambda$. By \Cref{keyrepprop} below, it suffices to show that $\psi(g) = \psi(g^a)$ for any $a$ prime to the order of $G$. Unraveling the definitions, it suffices to verify that if $M$ is an order-$d$ matrix in $\GL_2(\FF)$, then $\tr M = 0$ implies that $\tr M^a = 0$ for any $a$ prime to~$d$. But that is precisely \cref{tracezerolemma}; see \cref{lemma78section} below.

Hence we have rational numbers $q_1,\ldots, q_g$ so that 
\begin{equation}\label{eqn:psi}
\psi = q_1 \lambda_1 + \cdots + q_g \lambda_g.
\end{equation}

As in \eqref{eqn:plinear}, observe that the coefficient of \( p^{-s} \) behaves linearly, when multiplying the local factors at \( p \) of  $L$-functions which have an Euler product, namely:
\[
    \Big( 1 + \frac{b_1}{p^s} + \frac{b_2}{p^{2s}} + O({p^{-3s}}) \Big) \cdot \Big( 1 + \frac{b'_1}{p^s} + \frac{b'_2}{p^{2s}} + O(p^{-3s}) \Big) = 1 + \frac{b_1 + b'_1}{p^s} + O(p^{-2s}) \,. 
\]
Recall also the local factor of the Artin $L$-function $L(\rho_i, s)$ at a good prime $p$ is $$\frac{1}{\charpoly(\Frob_p|\rho_i)(p^{-s})} = \frac{1}{1 - \tr(\Frob_p| \rho_i)/p^{s} + O(p^{-2s}) } = 1 + \frac{\lambda_i(\Frob_p)}{p^{s}} + O(p^{-2s}).$$

Now consider \[
    J(s) = F(s)^{-1} \cdot \prod_{i=1}^g L_\Q(\rho_i, s)^{q_i} .
\]
Fix \( p \not\in T \).  The linearity above, and the fact that \( i_S(p) = \psi(\Frob_p) \), means that the local factor at \( p \) of \( J(s) \) satisfies
\begin{align*}
    J_p(s) &= \Big( 1 + \frac{\psi(\Frob_p) }{p^s} + O(p^{-2s}) \Big)^{-1} \cdot \prod_{i=1}^g \Big( 1 + \frac{\lambda_i(\Frob_p)}{p^s} + O(p^{-2s}) \Big)^{q_i} \\
    &= 1 + \frac{-\psi(\Frob_p)  + \sum_{i=1}^g q_i \lambda_i(\Frob_p)}{p^s} + O(p^{-2s}) 
     = 1 + O(p^{-2s}) \,.
\end{align*}
The vanishing of the \( p^{-s} \) coefficient here follows from \eqref{eqn:psi}.  From this, we deduce that \( \prod_{p\not\in T} J_p(s) \) converges absolutely at \( s = 1 \); this still holds after restoring the finitely many factors \( J_p(s) \), for \( p \in T \).

This completes the proof of \cref{artinlmain}. For \cref{dedekindzetamain} the argument is similar, except that we must show that $\psi$ is a $\QQ$-linear combination of permutation characters on cyclic subgroups (in other words, of inductions of the trivial representation from cyclic subgroups of $G$ up to $G$): see \cref{keyrepprop2}.
\end{proof}

\subsection{Representation-theory lemmas} In this section we recall a few facts from representation theory. Although we include some proofs for ease of reading, everything here is well known to experts. 

Fix $\Gamma$ be a finite group of order~$d$. Let $\Lambda := \{\lambda_1, \ldots, \lambda_g\}$ be the list of the irreducible characters of $\Gamma$, so that $\Gamma$ has~$g$ conjugacy classes. For any subring $R$ of $\CC$, let~$X(\Gamma, R)$ be the set of $R$-valued class functions of $\Gamma$, so that $X(\Gamma, R)$ is a free $R$-module of rank~$g$. Also let $R\Lambda$ be the $R$-linear span of $\Lambda$ inside $X(\Gamma, \CC)$. The following theorem is well known. 
\begin{Thm}[See, for example, Serre {\cite[Theorem 6]{SerreRep}}]\leavevmode
\begin{enumerate}[topsep = -5pt]
\item $\CC\Lambda =  X(\Gamma, \CC).$
\item If $K \subseteq \CC$ is a field, and every $\lambda$ in $\Lambda$ is in $X(\Gamma, K)$, then $K\Lambda = X(\Gamma, K)$.
\end{enumerate}
\end{Thm}

Less well known but completely straightforward is the following lemma, for which we recall that $\QQ(\mu_d)$ is the cyclotomic field generated by the $d^{\rm th}$ roots of unity $\mu_d$,
and that $\Gal(\QQ(\mu_d)/\QQ) \simeq (\ZZ/d\ZZ)^\times$, with $a \in (\ZZ/d\ZZ)^\times$ corresponding to $\sigma_a \in \Gal(\QQ(\zeta_d)/\QQ)$ characterized by $\sigma_a(\zeta) = \zeta^a$ for any $\zeta \in \mu_d$.

\begin{Prop}\label{galoisprop}\leavevmode

\crefformat{proplabel}{#2#1#3}
\begin{enumerate}[itemsep = 2pt]
    \item\label[proplabel]{gal} If $K/\QQ$ is Galois, then $\Gal(K/\QQ)$ acts on $X(\Gamma, K)$ on the left by conjugating values: that is, for $\sigma \in \Gal(K/\QQ)$, $\chi \in X(\Gamma, K)$, and $g \in \Gamma$, set  $${}^\sigma\!\chi(g) := \sigma(\chi(g)).$$
    \item The action from \cref{gal} preserves $\Lambda$ and hence $\QQ \Lambda$. 
    \item \label[proplabel]{conj} The group $(\ZZ/d\ZZ)^\times$ acts on the set of conjugacy classes of $\Gamma$, with \mbox{$a \cdot [g] := [g^a]$}, inducing an action of $(\ZZ/d\ZZ)^\times$ on $X(\Gamma, K)$ for any field $K$ by setting \mbox{${}^a \chi(g) := \chi(g^a)$}.
     \item \label[proplabel]{Qmud} Every $\lambda$ in $\Lambda$ is $\QQ(\mu_d)$-valued, so that  $\Lambda$ is a $\QQ(\mu_d)$-basis of $X\big(\Gamma, \QQ(\mu_d)\big)$.
    \item \label[proplabel]{galconj}\label[proplabel]{QLam}
    For any $\lambda \in \QQ\Lambda$, the action from \cref{gal} coincides with the action from \cref{conj} under the identification $\Gal(\QQ(\mu_d)/\QQ) = (\ZZ/d\ZZ)^\times$: that is for $a \in (\ZZ/d\ZZ)^\times$ we have
    $${}^{\sigma_a} \lambda(g) = \lambda(g^a).$$
\end{enumerate}
\end{Prop}

\begin{proof} Most statements are straightforward. For \cref{Qmud}, since every element of $\Gamma$ has order dividing $d$, it follows that every eigenvalue of every element in the image of any finite-dimensional representation of $\Gamma$ is a $d^{\rm}$ root of unity, so that traces are  $\QQ(\zeta_d)$-valued. Part~\cref{galconj} also follows by observing the effect of Galois conjugation on these eigenvalues, for example after diagonalizing. See also \cite{gannonGalois} for a discussion of these ideas.
\end{proof}

\crefformat{proplabel}{#2Proposition \ref*{galoisprop}#1#3}

We now prove a converse to \cref{QLam}. For an alternate exposition, see, for example, \cite[Theorem 29]{SerreRep}.
\begin{Prop}\label{keyrepprop} 
Let $\chi$ be in $X(\Gamma, \QQ)$, and assume that ${}^a \chi = \chi$ for any $a \in (\ZZ/d\ZZ)^\times$ under the action from \cref{conj}. Then $\chi \in \QQ\Lambda.$
\end{Prop}
\begin{proof} Let $K = \QQ(\mu_d)$. Since $K\Lambda = X(\Gamma, K)$ by \cref{Qmud}, we can write
\[
    \chi = \sum k_i \lambda_i \,.
\]
Since $\chi$ is $\QQ$-valued, for any $\sigma \in \Gal(K/\QQ)$, we have, on one hand,  
\begin{equation}\label{one} \chi = {}^\sigma\! \chi = \sum_i \sigma(k_i)\, {}^\sigma\! \lambda_i.\end{equation}
On the other hand, for any $a \in (\ZZ/d\ZZ)^\times$ we have, by assumption, 
\begin{equation}\label{two}\chi = {}^a \! \chi = \sum_i k_i\, {}^a \! \lambda_i. \end{equation}

But any $\sigma \in \Gal(K/\QQ)$ is equal to $\sigma_a$ for some $a \in (\ZZ/d\ZZ)^\times.$ Equating \eqref{one} and \eqref{two} for this choice of $a$ and using \cref{galconj}, we obtain $\sum_i k_i\, {}^a\! \lambda_i = \sum_i \sigma_a(k_i)\, {}^a \!\lambda_i$. Since the $\lambda_i$ are linearly independent, we conclude that $\sigma_a(k_i) = k_i$ for each $i$ and for every $a \in (\ZZ/d\ZZ)^\times$, so that the $k_i$ are rational and $\chi \in \QQ\Lambda$, as claimed. 
\end{proof}

\begin{Prop}[Artin induction theorem]\label{keyrepprop2}
Under the same assumption on $\chi$ as in \cref{keyrepprop}, the class function $\chi$ is a $\QQ$-linear combination of inductions of characters of the trivial representation on cyclic subgroups.
\end{Prop}

Although Artin's theorem is well known, this particular formulation is not obviously available in the literature. Therefore we include a full proof below. See also \mbox{\cite[p.~23]{Bartel}} for a similar circle of ideas.

\begin{Lem} 
\label{ind}
Let $G$ be a finite group and $H$ a subgroup. For $g \in G$ the trace $\tr \Ind_H^G \one (g)$ is a nonnegative integer, nonzero if and only if $H$ contains a conjugate of $g$.
\end{Lem}

\begin{proof}[Proof of \cref{ind}]
The representation $\Ind_H^G \one$ is the permutation representation on the left cosets $G/H$ of $H$ in $G$, where $G$ acts on $G/H$ by left translation. Since the trace of a permutation representation is the number of fixed points, it follows that 
\begin{align*}
    \tr \Ind_H^G \one (g) & = \#\{xH \in G/H: gxH = xH\}\\ &= \frac{\#\{x \in G: x^{-1} g x \in H \}}{\#H}
    = \frac{\#\{x \in G: x^{-1} \langle g\rangle x \in H \}}{\#H}.
\end{align*}
Here the second equality follows from the observation that if $x$ is an element  for which~$x^{-1}$ conjugates $g$ into $H$, then the entire coset $xH$ has the same property. \end{proof}

\begin{proof}[Proof of \cref{keyrepprop2}]
Let $H_1, \ldots, H_t$ be the complete set of cyclic subgroups of $G$ up to conjugacy, ordered by increasing size. For each $i$, let $\phi_i:= \tr \Ind_{H_i}^G \one.$
Also consider the equivalence relation on $G$ given by $$g \sim h  \text{ iff $g$ and $h$ generate conjugate cyclic subgroups of $G$}.$$
In general this is a coarser relation than conjugacy; we will call it \emph{power-conjugacy}. 
 It is not difficult to see that $G$ has exactly $t$ power-conjugacy classes; for each $H_i$ choose a generator $h_i$; let $P_i$ be the characteristic function of the power-conjugacy class of~$h_i$. 

The assumptions on $\chi$ imply that $\chi$ is a $\QQ$-linear combination of the $P_i$. We show that each $P_i$ is a $\QQ$-linear combination of the $\phi_i$. Indeed, \cref{ind} already implies that (conversely) each $\phi_i$ is a $\ZZ$-linear combination of the $P_i$. More precisely, if we write 
$\phi_i = \sum_{j = 1}^t b_{ij} P_j,$
then $$b_{ij} = 
\frac{\#\{x \in G: x^{-1} H_j x \subseteq H_i\}}{\#H_i}
.$$
In particular, $b_{ij} \neq 0$ only if a conjugate of $H_j$ is contained in $H_i$, which, because of the ordering assumption on $H_1, \ldots, H_t$, can only happen if $j \leq i$. Furthermore, $b_{ii}$ is always nonzero. It follows that the $M_t(\ZZ)$-matrix $(b_{ij})_{i,j}$ is upper-triangular with nonzero entries on the diagonal, so that is it is in $\GL_t(\QQ)$. Therefore each $P_j$ is a $\QQ$-linear combination of the $\phi_i$, as required.

An equivalent formulation to what we have shown: a $\QQ$-valued class function is a $\QQ$-linear combination of the $\phi_i$ if and only if it is constant on power-conjugacy classes.
\end{proof}

\subsection{A matrix lemma} \label{lemma78section}

\begin{Lem} \label{tracezerolemma}
Let $\FF$ be a finite field,  $M \in \GL_2(\FF)$ an order-$d$ matrix, and $a$ an integer relatively prime to $d.$ If $\tr M = 0$, then $\tr M^a = 0$ as well.
\end{Lem}
\begin{proof}
A $2 \times 2$ matrix $M$ satisfies the Cayley-Hamilton identity 
\begin{equation}\label{CH}M^2 - (\tr M) M + \det M = 0.\end{equation}
Multiplying by $M^k$ and taking traces gives us the recurrence satisfied by the sequence $\{\tr M^k\}_{k \geq 0}$: 
$$\tr M^{k + 2} = (\tr M) \tr M^{k+1} - (\det M) \tr M^k.$$

In particular, suppose $\tr M = 0$. Then \begin{equation}\label{rec}\tr M^{k+2} = - (\det M) \tr M^k \quad \text{for all $k \in \ZZ$}.\end{equation} If $\FF$ has characteristic 2, then since $\tr M^0 = \tr I_2 = 2 = 0$, we actually have $\tr M^k = 0$ for all $k \in \ZZ$, so that there is nothing to show. Otherwise, $\FF$ has odd characteristic, and~\eqref{rec} implies that $\tr M^k = 0$ if and only if $k$ is odd. It therefore suffices to show that the multiplicative order $d$ of trace-zero matrix $M$ is even. Indeed, Cayley-Hamilton \eqref{CH} again gives us $M^2 = -
\det M$. Raising this to the $d^{\rm th}$ power gives us $$1 = M^{2d} = (-1)^d (\det M)^d = (-1)^d.$$
Since $1 \neq -1$, we must have $d$ even as claimed.
\end{proof}

\appendix

\section*{Appendix: \texttt{Pari/GP} code and usage instructions}
\renewcommand{\thesection}{A}
\label{app:code}

All code and code exerts in this appendix were run with \texttt{Pari/GP} (version 2.15.4), cf.\,\cite{PARI2}.

\subsection{Implementation of the Ihara bounds for \texorpdfstring{\( \gamma_K \)}{gamma\_K}}
\label{app:iharacode}

The following code implements the calculation of Ihara's bounds.  The Dedekind-Kummer theorem is used in \mbox{\texttt{phiK(x, nf)}} to compute the factorization of \( p \mathcal{O}_K \) into prime ideals of \( K \), outside of primes dividing the discriminant of the defining polynomial (i.e. not dividing \( \operatorname{disc} \Z[\alpha] \)).  For primes dividing the discriminant of the defining polynomial, we rely on \texttt{Pari/GP}'s routine \texttt{idealprimedec} to correctly find the complete decomposition (although we only need the ramification index and residual index of the results).

\begin{lstlisting}[language=gp,belowcaptionskip=\medskipamount,caption={{Implementation of Ihara's bounds from \cref{prop:ihara}\llap{\raisebox{-1em}{\phantom{X}}}}}]
iharaBounds(x, nf) = {
  rK = 1/2 * log(abs(nf.disc)) - ( nf.r1/2 * (Euler + log(4*Pi))
         + nf.r2 * (Euler + log(2*Pi)) );

  lres = iharaEll(x, nf);
  phires = phiK(x, nf);
  sx = sqrt(x); lx = log(x);

  \\ [lower, upper]  
  [ (sx-1)/(sx+1) * (lx - phires + lres) - 1 - 2/(sx+1)*rK,
    (sx+1)/(sx-1) * (lx - phires + lres) - 1 + 2/(sx-1)*rK ]
};

iharaEll(x, nf) = {
  nf.r1 * ( log((x+1)/(x-1)) + 2/(x-1) * log((x+1)/2) ) 
   + nf.r2 * ( log(x/(x-1)) + 1/(x-1) * log(x) )
};

phiK(x, nf) = {
  nfpol = nf.pol; disc = poldisc(nf.pol);
  lx = log(x); 
  ttl = 0;

  \\ if p divides polynomial discriminant, compute factorisation
  \\ explicitly otherwise, use Dedekind-Kummer theorem
  forprime(p=2, x, 
    P = if(disc % p == 0, 
      apply(i->i.f, idealprimedec(nf, p)),    
      apply(i->poldegree(i[1]), Vec(factor(Mod(1, p)*nfpol)~ ))
     );
      
    lp = log(p);
    for(j=1, #P, 
      ub = lx/lp/P[j];
      ttl += sum(k=1, ub, (x/p^(k*P[j]) - 1) * P[j]) * lp;
    )
  );
  
  ttl * 1/(x-1)
};
\end{lstlisting}

\subsection{Usage instructions}\label{app:usage}
To define the number field \( E = \Q(\alpha) \) in \texttt{Pari/GP}, which is obtained by adjoining a root of \( \alpha \) of some irreducible  polynomial\footnote{This is the \( \mathfrak{S}_3 \)-fixed field of the degree 24 field appearing in \cref{Haberland59}} \( f_E(X) = X^4 + X^3 - 7X^2 - 11 X + 3 \), pass the defining polynomial \texttt{pol} to the routine \texttt{nfinit(pol)}.  One can then compute Ihara's bounds \eqref{eqn:ihara} for \( E \) at \( x = 10^8 \) by calling \texttt{iharaBounds(x, nf)} with these parameters, as follows:
\begin{lstlisting}[language=gpin]
<@\inpr@> nfE = nfinit(X^4 + X^3 - 7*X^2 - 11*X + 3);
<@\inpr@> iharaBounds(10^8, nfE)
\end{lstlisting}
\begin{lstlisting}[language=gpout]
<@\outpr@> [0.83356048096571271277339561065017380005, 0.83453128871901233112044719586629487940]
\end{lstlisting}
This means that \( \gamma_E \) is contained in the following interval 
\[
    \gamma_E \in (0.83356048\ldots, 0.83453128\ldots) \,,
\]
so that \( \gamma_E = 0.83\ldots \), accurate to 2 decimal places.

\subsection{Higher precision, using \texorpdfstring{$L$}{L}-function heuristics}
\label{app:higherprec}

An approach by Dokchitser \cite{Dokchitser} gives a way to numerically evaluate $L$-functions and their derivatives in a practical but \emph{non-rigorous} and \emph{uncertified} way.  Dokchitser notes \cite[Remarks 5.1--5.4]{Dokchitser} that in order to obtain rigorous algorithm, one needs to resolve various issues involving sufficient working precision, determining truncation points of infinite series and asymptotic expansions, and solving questions around the convergence of certain continued fraction expansions.  Nevertheless, the heuristics for this approach are implemented in \texttt{Pari/GP}, allowing one to obtain \emph{uncertified}\footnote{Again, see \textbf{Important Caveat} in the documentation on\newline \url{https://pari.math.u-bordeaux.fr/dochtml/html-stable/_L_minusfunctions.html}} numerical results to high precision.  For a number field \( E \), one can read \( \gamma_E \) off as the constant coefficient in the Laurent series of \( \zeta_E'/\zeta_E(s) \) around \( s = 1 \), as follows:
\begin{lstlisting}[language=gpin]
<@\inpr@> nfE = nfinit(X^4 + X^3 - 7*X^2 - 11*X + 3);
<@\inpr@> zetaE = lfuncreate(nfE); 
<@\inpr@> polcoef(lfun(zetaE, 1+s, 1)/lfun(zetaE, 1+s),0)
\end{lstlisting}
\begin{lstlisting}[language=gpout]
<@\outpr@> 0.83393401339070668873643457029774979792
\end{lstlisting}
Assuming that \texttt{Pari/GP} has correctly evaluated these $L$-functions, we expect that
\[
    \gamma_E \overset{?}{=} 0.8339340133907066887364345\ldots \,,
\]
with significantly higher precision that the value obtained with \texttt{iharaBounds(x, nf)}.  This result is \emph{uncertified}, but does indeed agree with \( \gamma_K = 0.83\ldots \) as obtained above using Ihara's bounds \eqref{eqn:ihara}.

\subsection{Regulators}
\label{app:reg}
To compute the regulator of a number field one can use the \texttt{bnfinit} routine from \texttt{Pari/GP} to define a (Buchmann's) number field, a data structure which contains the resulting number field, and deeper invariants thereof (such as the unit group and the class group).  The regulator is then available via \texttt{bnf.reg}.  For the field \( E \) with defining polynomial \( f_E(X) = X^4 + X^3 - 7X^2 - 11X + 3 \), this is done as follows.
\begin{lstlisting}[language=gpin]
<@\inpr@> bnfE = bnfinit(X^4 + X^3 - 7*X^2 - 11*X + 3);
<@\inpr@> bnfE.reg
\end{lstlisting}
\begin{lstlisting}[language=gpout]
<@\outpr@> 20.522585993622250450776428334525332118
\end{lstlisting}
By default the result of \texttt{bnfinit} relies on GRH for this number field, in order to apply Buchmann's sub-exponential algorithm to compute the idea class group, fundamental units, and regulator.  One can determine whether the GRH assumption can be removed via \texttt{bnfcertify(bnf)} (with result 1 if and only if GRH can be removed).  If the GRH assumption can be removed one hence obtains \emph{certified} results for the class group, the fundamental units and in particular the regulator.
\begin{lstlisting}[language=gpin]
<@\inpr@> bnfE = bnfinit(X^4 + X^3 - 7*X^2 - 11*X + 3);
<@\inpr@> bnfcertify(bnfE)
\end{lstlisting}
\begin{lstlisting}[language=gpout]
<@\outpr@> 1
\end{lstlisting}

For the number fields in \cref{Haberland59}, the \texttt{bnfcertify} verification on \( E \) and \( M \) is nearly instantenous, while on \( L \), the degree 24 splitting field of \( E \), whose smallest defining polynomial can be computed via \texttt{polredabs(nfsplitting(pol))}, this verification takes approximately 1 hour.
\begin{lstlisting}[language=gpin]
<@\inpr@> bnfL = bnfinit(polredabs(nfsplitting(X^4 + X^3 - 7*X^2 - 11*X + 3)));
<@\inpr@> bnfcertify(bnfL)
\end{lstlisting}
\begin{lstlisting}[language=gpout,belowskip=-5pt]
<@\outpr@> 1
\end{lstlisting}
\begin{lstlisting}[language=gpin]
<@\inpr@> ## \\ displays runtime of the last command
\end{lstlisting}
\begin{lstlisting}[language=gpoutmsg]
  ***   last result: cpu time <@\color{olive}55min, 18,276 ms@>, real time <@\color{olive}55min, 18,322 ms@>.

\end{lstlisting}

\parskip=0pt           


\begin{thebibliography}{99}
\bibitem{Ahlgren} S.~Ahlgren, The points of a certain fivefold over finite fields and the twelfth power of the eta function, 
\textit{Finite Fields Appl.} \textbf{8} 
(2002) 18--33.


\bibitem{Bartel}A.~Bartel with F.~Bouyer,  Further Representation Theory, unpublished notes, 2014.          \url{https://warwick.ac.uk/fac/sci/maths/people/staff/fbouyer/representation_theory.pdf}



\bibitem{BC} K.~Belabas and H.~Cohen, \textit{Numerical algorithms for number theory—using Pari/GP}, Mathematical Surveys and Monographs 
\textbf{254}, American Mathematical Society, Providence, RI, 2021.

\bibitem{Bim} J. Bellaïche, \href{https://doi.org/10.1016/j.aim.2019.07.001}{Image of pseudo-representations and coefficients of modular forms modulo $p$}, \emph{Advances in Mathematics}, Volume \textbf{353}, 2019, pp. 647--721.


\bibitem{bemo} B.~Berndt and P.~Moree, Sums of two squares and the tau function: Ramanujan's trail, 
\url{https://arxiv.org/abs/2409.03428}.

\bibitem{bono} B.C. Berndt and K. Ono, Ramanujan's unpublished manuscript 
on the partition and tau functions with proofs and commentary, {\it The
Andrews Festschrift} (Maratea, 1998), (Eds.) D. Foata and
G.N. Han, 2001, 39--110.

\bibitem{Boylan} M.~Boylan, Exceptional congruences for the coefficients of certain eta-product newforms, 
\textit{J. Number Theory} \textbf{98} 
(2003), 377--389. 


\bibitem{Brighton} J.W.S. Cassels and A. Fr\"ohlich (Eds.), 
\textit{Algebraic Number Theory},
Proceedings of an instructional conference organized by the London Mathematical Society, Academic Press, London; Thompson Book Co., Inc., Washington, DC, 1967. 

\bibitem{C} S. Chowla, L-series and elliptic curves, 
\textit{Number theory day} (Proc. Conf., Rockefeller Univ., New York, 1976), 
\textit{Lecture Notes in Math.} \textbf{626}, Springer, Berlin-New York, 1977, 
1--42.


\bibitem{CC} S. Chowla and M. Cowles, 
On the coefficients $a_n$ in the expansion\\ $x\prod_{n=1}^{\infty}(1-x^n)^2(1-x^{11n})^2=\sum_{n=1}^{\infty}a_nx^n,$
\textit{J. Reine Angew. Math.} \textbf{292} (1977), 115--116.

\bibitem{CLM} A.\,Ciolan, A.\,Languasco and P.\,Moree, 
Landau and Ramanujan approximations for divisor sums and coefficients of cusp forms,
\textit{J.\,Math.\,Anal.\,Appl.} \textbf{519} (2023), no. 2, Paper No. 126854.

\bibitem{ibid} Ibid., results webpage, \url{https://www.dei.unipd.it/~languasco/CLM.html}

\bibitem{cogdell} J.~W.~Cogdell, \href{https://people.math.osu.edu/cogdell.1/Artin-www.pdf}{$L$-functions and non-abelian class field theory, from Artin to Langlands}. In:
D.~Dumbaugh, J.~Schwermer, 
\emph{Emil Artin and Beyond --- Class Field Theory and $L$-functions}.
With contributions by James Cogdell and Robert Langlands
Herit. Eur. Math.
European Mathematical Society (EMS), Zürich, 2015.

\bibitem{Collins} M.J.\,Collins, \textit{Representations and characters of finite groups},
Cambridge Studies in Advanced Mathematics \textbf{22}, Cambridge University Press, Cambridge, 1990. 

\bibitem{Cowles} J.\,Cowles, Some congruence properties of three well-known sequences: 
Two notes, \textit{J. Number Theory} \textbf{12} (1980), no. 1, 84--86.

\bibitem{Deligne} P. Deligne, Formes modulaires et répresentations $\ell$-adiques, Sém. Bourbaki 1968-1969, exposé \textbf{355}, \emph{Lecture Notes in Mathematics}, \textbf{179}, Springer-Verlag, (1971), pp. 139--172.

\bibitem{DeligneSerre} P.~Deligne, J.-P.~Serre, 
\href{http://www.numdam.org/item?id=ASENS_1974_4_7_4_507_0}{Formes modulaires de poids 1.}
\emph{Ann. Sci. École Norm. Sup.} (4) \textbf{7} (1974), pp.~507–530.

\bibitem{DS} F.~Diamond, J.~Shurman,
\emph{A First Course in Modular Forms.}
Graduate Texts in Mathematics {\bf 228},
Springer-Verlag, New York, 2005.


\bibitem{Dokchitser} T.~Dokchitser, Computing special values of motivic $L$-functions.
Experiment. Math.13(2004), no.2, 137–149.

\bibitem{edwards} H.M.~Edwards, {\it Fermat's Last Theorem},  Graduate
Texts in Mathematics {\bf 50}, Springer-Verlag, New York, 1977.

\bibitem{EH} T.~Evink and P.A.~Helminck, Tribonacci numbers and primes of the form $p=x^2+11y^2$, 
\textit{Math. Slovaca} \textbf{69}(3) (2019) 521--532.

\bibitem{FordLucaMoree} 
K.~Ford, F.~Luca and P.~Moree, 
 Values of the Euler $\phi$-function not divisible by
a given odd prime, and the distribution of Euler-Kronecker constants for
cyclotomic fields, \textit{Math. Comp.} \textbf{83} (2014), no. 287, 1447--1476. 

\bibitem{FT} A.~Fr\"ohlich and M.J.~Taylor, 
\textit{Algebraic number theory},
Cambridge Studies in Advanced Mathematics \textbf{27}, Cambridge University Press, Cambridge, 1993.

\bibitem{gannonGalois} T. Gannon, 
The Galois action on character tables,  
\textit{Groups and symmetries}, 165--172,
\textit{CRM Proc. Lecture Notes} \textbf{47},
American Mathematical Society, Providence, RI, 2009


\bibitem{Gelbart} S.~Gelbart, Elliptic 
curves and automorphic representations,
\textit{Advances in Math.} 
\textbf{21} (1976), 235--292.



\bibitem{Haberland} K. Haberland, Perioden von Modulformen einer Variabler and Gruppencohomologie. I, II, III, 
Math. Nachr. \textbf{112} (1983), 245--282, 283--295, 297--315.



\bibitem{HIKW} Y. Hashimoto, Y. Iijima, N. Kurokawa and M. Wakayama, Euler's constants for the Selberg and the Dedekind $\zeta$-functions, \textit{Bull. Belg. Math. Soc. Simon Stevin} \textbf{11} (2004), 493--516.

\bibitem{HIS} D.~Hu, H.R.~Iyer and A.~Shashkov,  Modular forms and an explicit Chebotarev
variant of the Brun–Titchmarsh theorem, 
\textit{Res. Number Theory} \textbf{9} (2023), no. 3, Paper No. 46, 37 pp.

\bibitem{Ihara} Y.~Ihara,  
On the Euler-Kronecker constants of global fields and primes with small norms.  In: Algebraic geometry and number theory, 407–451.
Progr. Math., 253
Birkhäuser Boston, Inc., Boston, MA, 2006.

\bibitem{Ito} H. Ito, 
A remark on a theorem of Chowla-Cowles, 
\textit{J. Reine Angew. Math.} \textbf{332} (1982), 
151--155.

\bibitem{Joyner} D. Joyner, 
\textit{Distribution theorems of 
$L$-functions},
Pitman Research Notes in Mathematics Series 
\textbf{142}, Longman Scientific \& Technical, Harlow; John Wiley \& Sons, Inc., New York, 1986.


\bibitem{KLW} C.B. Khare, A.F. la Rosa 
and G. Wiese, Splitting fields of $X^n-X-1$ (particularly for $n=5$), prime decomposition and modular forms, \textit{Expo. Math.} 
\textbf{41} (2023), 475--491.

\bibitem{KiVe} I.~Kiming and H.A.~Verrill, On modular mod $l$ Galois representations with exceptional images, 
\textit{J. Number Theory} 
\textbf{110} (2005), 236--266.

\bibitem{Klein} F.~Klein, \textit{Vorlesungen \"uber die Theorie der elliptischen Modulfunktionen. Band II: Fortbildung und Anwendung der Theorie},
Ausgearbeitet und vervollst\"andigt von Robert Fricke, Nachdruck der ersten Auflage, Bibliotheca Mathematica Teubneriana \textbf{11}, Johnson Reprint Corp., New York; 
B. G. Teubner Verlagsgesellschaft, Stuttgart, 1966.



\bibitem{Ale} A.~Languasco, Efficient computation of the Euler-Kronecker constants of prime cyclotomic fields, 
\textit{Res. Number Theory} \textbf{7} (2021), no. 1, Paper No. 2, 22 pp.

\bibitem{Aleunified} A.~Languasco, A unified strategy to compute some special functions of number-theoretic interest, 
\textit{J. Number Theory} \textbf{247} (2023), 118--161.



\bibitem{LaMoree} A.~Languasco and P.~Moree,
Euler constants from primes in arithmetic progression, 
to appear in \textit{Math.\,Comp.}, \url{https://arxiv.org/abs/2406.16547}.

\bibitem{LR} A.~Languasco and L.~Righi, A fast algorithm to compute the Ramanujan-Deninger gamma function and some number-theoretic applications, \textit{Math. Comp.} 
\textbf{90} (2021), 2899--2921.

\bibitem{LT} S.~Lang and 
H.~Trotter, \textit{Frobenius distributions in $GL_2$-extensions}, Distribution of Frobenius automorphisms in $GL_2$-extensions of the rational numbers. 
\textit{Lecture Notes in Mathematics} \textbf{504}, Springer-Verlag, Berlin-New York, 1976. 


\bibitem{LMFDB} $L$-functions and Modular Forms Database (LMFDB).  \url{https://www.lmfdb.org/}


\bibitem{mazur} B.~Mazur, Modular curves and the Eisenstein ideal, 
\textit{Inst. Hautes \'Etudes Sci. Publ. Math.} \textbf{47} (1977), 33--186.

\bibitem{error} B. Mazur, 
Finding meaning in error terms, 
\textit{Bull. Amer. Math. Soc.} (N.S.) \textbf{45} (2008), 185--228.


\bibitem{Milne} J.S.~Milne, \textit{Class Field Theory},
\url{https://www.jmilne.org/math/CourseNotes/cft.html}

\bibitem{MRama} P.~Moree, On some claims in Ramanujan's `unpublished'
manuscript on the partition and tau functions, 
{\it Ramanujan J.} \textbf{8} (2004), 317--330.

\bibitem{Mpreprint} P.~Moree, Values of the Euler phi function not divisible by a prescribed odd prime,
unpublished preprint (precursor of \cite{FordLucaMoree}) available at \url{https://arxiv.org/abs/math/0611509}.

\bibitem{India} P.~Moree, Counting numbers in multiplicative sets: Landau versus Ramanujan, {\it Mathematics Newsletter}, 
\#3 (2011), 73--81.

\bibitem{MC} P.~Moree and J.~Cazaran, On a claim of Ramanujan in
his first letter to Hardy, {\it Expos. Math.} \textbf{17} (1999),
289--312.

\bibitem{MN} P.~Moree and A.~Noubissie, Higher reciprocity laws 
and ternary linear recurrence sequences,
\url{https://arxiv.org/abs/2205.06685}, \textit{Int. J. Number Theory}, to 
appear.

\bibitem{Murtysieving} M.R.~Murty, 
Sieving using Dirichlet series, 
\textit{Currents trends in number theory (Allahabad, 2000)}, 
111--124, Hindustan Book Agency, New Delhi, 2002.

\bibitem{Odoni1} R.W.K.~Odoni,
Solution of some problems of Serre on modular forms: The method of Frobenian functions, 
\textit{Recent progress in analytic number theory}, Symp. Durham 1979, Vol. 2 (1981), 159-169 (1981). 


\bibitem{Odoni2} R.W.K.~Odoni, Notes on the method of Frobenian functions with applications to Fourier coefficients of modular forms,
\textit{Elementary and analytic theory of numbers}, 
Banach Cent. Publ. \textbf{17} (1985), 371--403.

\bibitem{PARI2}
The PARI~Group, Pari/GP version \texttt{2.15.4}, Univ. Bordeaux, 2023,
\url{http://pari.math.u-bordeaux.fr/}.

\bibitem{RamMurty} M.~Ram Murty, An introduction to Artin L-functions, 
\textit{J. Ramanujan Math. Soc.} \textbf{16} (2001), 261--307.


\bibitem{Rankin} R.A.~Rankin, The divisibility of divisor functions.
\textit{Proc. Glasgow Math. Assoc.} \textbf{5} (1961), 
35--40.

\bibitem{Rankin76} R.A. Rankin, 
Ramanujan's unpublished work on congruences, 
\textit{Modular functions of one variable. V} (Proc. Second Internat. Conf., Univ. Bonn, Bonn, 1976), pp. 3--15, 
\textit{Lecture Notes in Math.} \textbf{601}, Springer, Berlin, 1977. 

\bibitem{RTTY} J. Rosen, Y. Takeyama, K. Tasaka
and S. Yamamoto,
The ring of finite algebraic numbers and its application to the law of decomposition of primes, \url{https://arxiv.org/abs/2208.11381}.

\bibitem{RS} J.B.~Rosser and L.~Schoenfeld, Approximate formulas for some functions of
prime numbers, \textit{Illinois Journal Math.} \textbf{6} (1962), 64--94.

\bibitem{SAGE} W.~A. Stein et~al., \emph{{S}age {M}athematics {S}oftware
({V}ersion 10.1)}, The Sage Development Team, 2023, {\tt
http://www.sagemath.org}.


\bibitem{serrecong} J.-P. Serre,
Congruences et formes modulaires [d'apr\`es H.P.F. Swinnerton-Dyer], 
S\'eminaire Bourbaki, 24\`eme ann\'ee (1971/1972), Exp. No. 416, pp. 319--338,
\textit{Lecture Notes in Math.} \textbf{317}, Springer, Berlin-New York, 1973.



\bibitem{serre} J-P. Serre, Divisibilit\'e de certaines fonctions
arithm\'etiques, {\it Enseignement Math.} {\bf 22} (1976), 227--260.

\bibitem{SerreRep} J-P. Serre, \emph{Linear Representations of Finite Groups}, {Springer}, GTM 42, 1977. 

\bibitem{serreIHES} J.-P. Serre,
Quelques applications du th\'eor\`eme de 
densit\'e de Chebotarev, 
\textit{Inst. Hautes \'Etudes Sci. Publ. Math.}
\textbf{54} (1981), 323--401.

\bibitem{serrejordan} J.-P. Serre, 
On a theorem of Jordan, 
\textit{Bull. Amer. Math. Soc. (N.S.)} \textbf{40} (2003), 429--440.


\bibitem{shimura} G. Shimura, 
A reciprocity law in non-solvable extensions, 
\textit{J. Reine Angew. Math.} \textbf{221} (1966), 209--220.



\bibitem{Antwerp} H.P.F. Swinnerton-Dyer, 
On $\ell$-adic representations and congruences for coefficients of modular forms. Modular functions of one variable, III (Proc. Internat. Summer School, Univ. Antwerp, Antwerp, 1972), pp. 1--55,
\textit{Lecture Notes in Math.} \textbf{350}, Springer, Berlin-New York, 1973.

\bibitem{S-D} H.P.F. Swinnerton-Dyer, Congruence properties of $ \tau(n),$ in \textit{Ramanujan Revisited} (Urbana-Champaign, Ill., 1987), 289--311, Academic Press, Boston, MA, 1988.


\end{thebibliography}
\end{document}